\documentclass[reqno,10pt]{amsart}
\usepackage[english]{babel}
\usepackage{amsmath}

\usepackage{amssymb}
\usepackage[scr=boondoxo]{mathalfa}
\usepackage{graphicx}
\usepackage{setspace}
\usepackage{epstopdf}

\usepackage{subfigure}
\usepackage{color}

\newcommand{\bpr}{\begin{trivlist} \item[]{\bf Proof. }}
\newcommand{\epr}{\hspace*{\fill} $\qed$\end{trivlist}}
\newcommand\response[1]{{\color{black}{#1}}}

\newcommand{\be}{\begin{eqnarray}}
\newcommand{\ee}{\end{eqnarray}}
\newcommand{\ba}{\begin{align}}
\newcommand{\ea}{\end{align}}
\newcommand{\bi}{\begin{itemize}}
\newcommand{\ei}{\end{itemize}}

\newcommand{\secref}[1]{Section~\ref{sec:#1}}

\newcommand{\seclab}[1]{\label{sec:#1}}
\newcommand{\eqlab}[1]{\label{eq:#1}}
\renewcommand{\eqref}[1]{(\ref{eq:#1})}

\newcommand{\figref}[1]{Fig.~\ref{fig:#1}}
\newcommand{\figlab}[1]{\label{fig:#1}}

\newcommand{\lemmaref}[1]{Lemma~\ref{lemma:#1}}
\newcommand{\lemmalab}[1]{\label{lemma:#1}}
\newcommand{\remref}[1]{Remark~\ref{remark:#1}}
\newcommand{\remlab}[1]{\label{remark:#1}}
\newcommand{\corref}[1]{Corollary~\ref{cor:#1}}
\newcommand{\corlab}[1]{\label{cor:#1}}
\newcommand{\thmref}[1]{Theorem~\ref{theorem:#1}}
\newcommand{\thmlab}[1]{\label{theorem:#1}}
\newcommand{\tablab}[1]{\label{tab:#1}}
\newcommand{\tabref}[1]{Table~\ref{tab:#1}}

\newcommand{\appref}[1]{Appendix~\ref{app:#1}}

\newcommand{\applab}[1]{\label{app:#1}}

\newcommand{\C}{\mathbb C}

\newcommand{\R}{\mathbb R}

\newtheorem{theorem}{Theorem}[section]
\newtheorem{proposition}[theorem]{Proposition}

\newtheorem{lemma}[theorem]{Lemma}
\newtheorem{cor}[theorem]{Corollary}

\newtheorem{remark}[theorem]{Remark}

\numberwithin{equation}{section}
\usepackage{enumitem}
\begin{document}
%
\title{A dynamical systems approach to WKB-methods: The simple turning point}

\author {K. Uldall Kristiansen and P. Szmolyan} 
\date\today
\maketitle

\vspace* {-2em}
\begin{center}
\begin{tabular}{c}
\end{tabular}
\end{center}

 \begin{abstract}
 In this paper, we revisit the classical linear turning point problem for the  second order differential equation $\epsilon^2 x'' +\mu(t)x=0$ with $\mu(0)=0,\,\mu'(0)\ne 0$ for $0<\epsilon\ll 1$. Written as a first order system, $t=0$ therefore corresponds to a turning point connecting hyperbolic and elliptic regimes. Our main result is that we provide an alternative approach to WBK that is based upon dynamical systems theory, including GSPT and blowup, and we bridge -- perhaps for the first time -- hyperbolic and elliptic theories of slow-fast systems. As an advantage, we only require finite smoothness of $\mu$. The approach we develop will be useful in other singular perturbation problems with hyperbolic--to--elliptic turning points. 
 \end{abstract}
 

 \section{Introduction}

\bigskip
In this paper, we reconsider the classical linear turning point problem for the  second order differential equation
\begin{equation}\eqlab{secondordereq}
\epsilon^2 x''+ \mu(t)x  = 0
\end{equation}
for a function $x(t)$, $t \in I \subset \R$ and  $0<\epsilon \ll 1$. 
On intervals where $\mu >0$, solutions are highly oscillatory whereas $\mu <0$ causes rapid  exponential decay/growth for $0<\epsilon\ll 1$. 
This is obvious in the case that $\mu$ is constant, but carries over to time-dependent $\mu(t)$. A turning point is a point $t_0$ where 
$\mu$ vanishes. We will assume that $t_0=0$ and focus on the most common situation of a simple zero, i.e.  $\mu(0)=0$, $\mu'(0)\ne 0$. 
The case $\mu(t) = t$ corresponds via the rescaling $\tau =  \epsilon^{-2/3} t$  to the famous Airy equation
\begin{align}
x''(\tau) =- \tau x(\tau),\eqlab{airy0}
\end{align} 
first studied in \cite{airypaper} in the context of problems from optics.\footnote{\label{foot} Traditionally the Airy equation actually takes the form $x''=\tau x$, see \cite{airypaper}; this form can be obtained from \eqref{airy0} by reversing time $\tau$. However, in our dynamical systems framework, where we will think of $\tau$ (and $t$) as time, it is more natural to work with \eqref{airy0} and for simplicity we will therefore throughout refer to the form \eqref{airy0} when talking about the Airy-equation.} 

\subsection{The Schr{\"o}dinger equation and WKB}
The analysis of  \eqref{secondordereq} has a long history due to its relevance for the eigenvalue problem for the
one-dimensional Schr{\"o}dinger equation
 \begin{align}
  \epsilon^2 x'' &=\left(V(t)-E\right)x,\eqlab{eq0}
 \end{align}
in the semi-classical limit $\epsilon \rightarrow 0$. Here $x(t)$ is the wave function, $t \in \R$ a spatial variable,  $V(t)$ the potential 
and  $E$  the energy. 
In the case of a potential well the eigenvalue problem is to find the values of $E$ for which solutions exist which decay as 
$t \to \pm \infty$. For $E < V$ solutions are exponentially growing/decaying, for $E >V$ solutions are oscillatory.
Turning points are points $t_0$ with $V(t_0) =E$. In the corresponding classical dynamics these points are the points where
the velocity of the corresponding particle changes its sign, hence the name turning point.

There exists a huge literature on the asymptotic analysis of \eqref{secondordereq}
and related more complicated linear differential equations with or without turning points. We refer to the classics \cite{bender1978,fedoryuk1993,olver1974,wasow1985a} 
for extensive treatments including the history of the subject. 
To put the approach and results of this paper into context we briefly sketch 
the basic formal approach known as the Liouville-Green or WKB-approximation. Later we will also comment on
rigorous variants of the WKB-method and other related asymptotic methods.

Away from turning points, solutions of \eqref{secondordereq} can be approximated by
the WKB-ansatz
 \begin{equation}\eqlab{wkbansatz} 
  x(t) =  \exp ( \epsilon^{-1} S_0(t) + S_1(t) + \epsilon S_2(t)   + \cdots )  
 \end{equation} 
 which leads to 
\[ S_0(t) = \pm \int_{t_0}^t \sqrt{-\mu(s)}\,ds, \quad S_1(t) = -\frac{1}{4}\ln (   |\mu(t) | ) \]
and the two corresponding WKB-solutions
\begin{equation}\eqlab{wkbsolution} 
  x_{\pm}(t) = \frac{1}{ \sqrt[4]{|\mu(t)|} } \exp   \left ( \pm \frac{1}{\epsilon}\int_{t_0}^t\sqrt{-\mu(s)}\,ds + \mathcal O(\epsilon)\right ).
 \end{equation}  
For $\mu(t) < 0$ WKB-solutions are real and exponentially growing/decaying. On the other hand, for 
$\mu(t) > 0$ real oscillatory WKB-solutions are obtained by separating into real and imaginary parts. The validity of the approximation \eqref{wkbsolution} 
away from turning points for $0<\epsilon \ll 1$ is also well known.
At turning points the approximation \eqref{wkbsolution} breaks down, due to the denominator vanishing  but also due to the singularity of the complex square root.
This leads to the so called connection problem: On one side of a turning point
a solution is approximated by a linear combination of two 
exponential WKB-solutions and on the other side by a linear combination of two oscillatory WKB-solutions. The connection problem is the task to relate these two 
different approximations across the turning point. 
Early on this was achieved in a formal way by replacing $\mu(t)$ by the linear
function $at$ with $
a:=\mu'(0) >0$ which is a reasonable approximation for $t$ close to the turning point $t_0 =0$. This and the rescaling 
\begin{equation}\eqlab{airyscaling} 
\tau = \epsilon^{-2/3}a^{1/3} t
\end{equation}
of $t$ reduces equation \eqref{secondordereq} to the Airy equation \eqref{airy0} for $\epsilon\rightarrow 0$.
The WKB-solutions \eqref{wkbsolution} are now rewritten on both sides of the
turning point in the scaled variable $\tau$
and matched to the known asymptotic behaviour of solutions of the Airy equation, for details of this formal procedure see \cite{bender1978,hall2013,olver1974}.

An additional salient feature of the connection problem is its  directionality: The exponentially growing WKB-solution to the left of a turning point $t_0$ with $\mu'(t_0) >0$
can be matched to a linear combination of the oscillatory WKB-solutions to the right of the turning point;  similarly the exponentially  decaying WKB-solution to the right 
of a turning point  $t_1$ with $\mu'(t_1) <0$
can be matched to a linear combination of the oscillatory WKB-solutions to the left of the turning point, see the discussion in \cite{olver1974}.

In the case of a potential well $V(t)$ with one minimum at $t=0$,  $V'(t) \neq 0$ for $t \neq 0$ and $\lim_{t \to \pm \infty} V(t) = \infty$ 
this is sufficient to solve the eigenvalue problem
for the eigenvalues $E$ of the Schr\"odinger equation \eqref{eq0} asymptotically.
In this situation we have that for any $E>0$, there exist two (and only two)  turning points  $t_-(E)<  t_+(E)$ 
corresponding to the solutions of the equation $V(t) -E =0$  see \figref{potentialV}. 
It is known, see e.g. \cite{bender1978,hall2013}, that the eigenvalues $E_n(\epsilon)$, $n\in \mathbb N_0$  that are $\mathcal O(1)$ with respect to $\epsilon$ (under certain assumptions) 
are approximated (to order $o(\epsilon)$) by solutions $E$ of the famous Bohr-Sommerfeld quantization condition:
\begin{align}
 \frac{1}{\epsilon} \int_{t_-(E)}^{t_+(E)} \sqrt{E-V(t)}dt = \pi\left(n+\frac{1}{2}\right). \eqlab{quant}
\end{align}
The fraction $1/2$ appearing in the right hand side of the quantization condition is known as the Maslow correction.
The quantization condition is obtained in the following way. The  exponentially decaying WKB-solution for $t \to -\infty$, existing in the ``classically forbidden region'' $(-\infty, t_-)$,  
is connected to an oscillatory solution $x_{-}^{\rm{osc}} $  in the ``classically allowed region'' $(t_-,t_+)$;  similarly,    the exponentially decaying WKB-solution for $t \to \infty$, existing in the classically forbidden region $(t_+,\infty)$, is connected to 
an oscillatory solution   $x_{+}^{\rm{osc}} $  in the ``classically allowed region'' $(t_-,t_+)$. The quantization condition follows then from the requirement 
 $x_{-}^{\rm{osc}} (t) = x_{+}^{\rm{osc}}(t) $,  $t \in  (t_-,t_+)$.

\begin{figure}[!ht] 
\begin{center}
{\includegraphics[width=.60\textwidth]{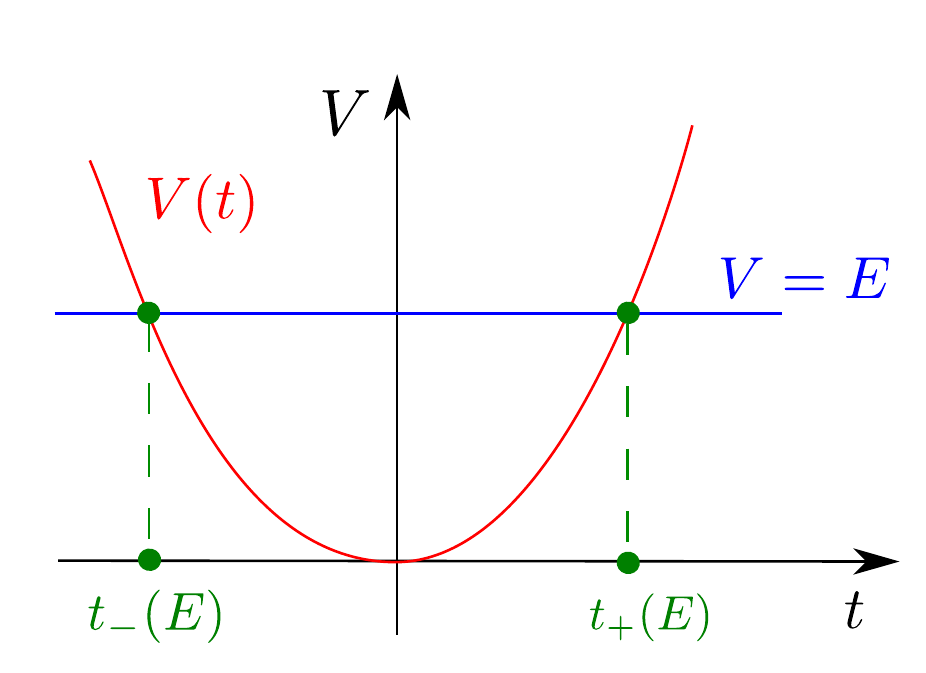}}
\end{center}
 \caption{A  potential well $V(t)$ having a global minimum at $t=0$. For each $E>0$ there exist two turning points $t_\pm(E)$ where $E-V(t)$ changes sign. }
\figlab{potentialV}
\end{figure}

\subsection{A dynamical systems view of the turning point problem}
This brief discussion of the formal WKB-approach shows that the main ingredient in solving the eigenvalue problem is to track genuine solutions corresponding
to formal exponential WKB-solutions across turning points and  to identify their continuation as genuine oscillatory solutions close to oscillatory WKB-solutions.

In this paper, we will give a detailed analysis of this basic problem in the framework  of dynamical systems theory.
 To this purpose we rewrite the  second order equation \eqref{secondordereq} as the first order system
\begin{equation}\eqlab{system1}
\begin{aligned}
 \dot x &=y,\\
 \dot y&=-\mu(t)x,\\
 \dot t&=\epsilon.
\end{aligned}
\end{equation}
We have added the equation for $t$ to make the system autonomous.  Central to our approach is the observation that
system \eqref{system1} is a slow-fast system for $0<\epsilon\ll 1$ with slow variable $t$ and fast variables $(x,y)$. 

Slow-fast systems have during the past two or three decades been successfully studied by Geometric Singular Perturbation Theory (GSPT), see \cite{fen3, jones_1995,krupa_extending_2001}.
  In short, GSPT is a collection of theories and methods for studying singularly perturbed ODEs using invariant manifolds. This includes, first and foremost, Fenichel's original theory \cite{fen3}, see also \cite{jones_1995,kuehn2015}, for the perturbation of compact normally hyperbolic critical manifolds and their stable and unstable manifolds. Besides the Exchange Lemma \cite{jones_1995,schecter2008a} and Entry-Exit functions \cite{de2016a,hayes2016a,hsu2017a}, GSPT nowadays also consists of the blowup method, following \cite{krupa_extending_2001,krupa_extending2_2001,krupa_relaxation_2001}, see also \cite{dumortier1996a}, as the key technical tool, allowing for an extension of Fenichel's theory near nonhyperbolic points. Moreover, although GSPT is based on hyperbolicity,  it has recently \cite{de2020a}, see also earlier work \cite{canalis-durand2000a,fruchard2012a,kristiansenwulff,wasow1965a}, been shown that slow manifolds, the central objects of GSPT, also exist in the elliptic setting under certain conditions including analyticity of the vector-field.

%

In the context of \eqref{system1},
the line $(0,0,t)$, $t \in I$ is a line of equilibria for $\epsilon=0$. This is called a critical manifold in GSPT.
The stability properties of points on this line change at the turning point $t=0$.
 Indeed, the linearization around any point $(0,0,t)$ produces eigenvalues $\pm \lambda(t)$ with $\lambda(t)=\sqrt{-\mu(t)}$ for $\epsilon=0$.
We assume that  $\mu$ changes sign at $t=0$ with $\mu'(0) >0$, hence the eigenvalues go (locally) from real to imaginary. The basic problem is then to describe the transition of solutions from the  hyperbolic side $t<0$   to the elliptic side $t>0$.

In the language of GSPT, the critical manifold is normally hyperbolic for $t  \leq  -\delta <0$ with unstable and stable manifolds $W^u$ and $W^s$, respectively,  obtained by  attaching the stable and unstable eigenspaces to each point, see  \figref{xyt}.
  \begin{figure}[!ht] 
\begin{center}
{\includegraphics[width=.85\textwidth]{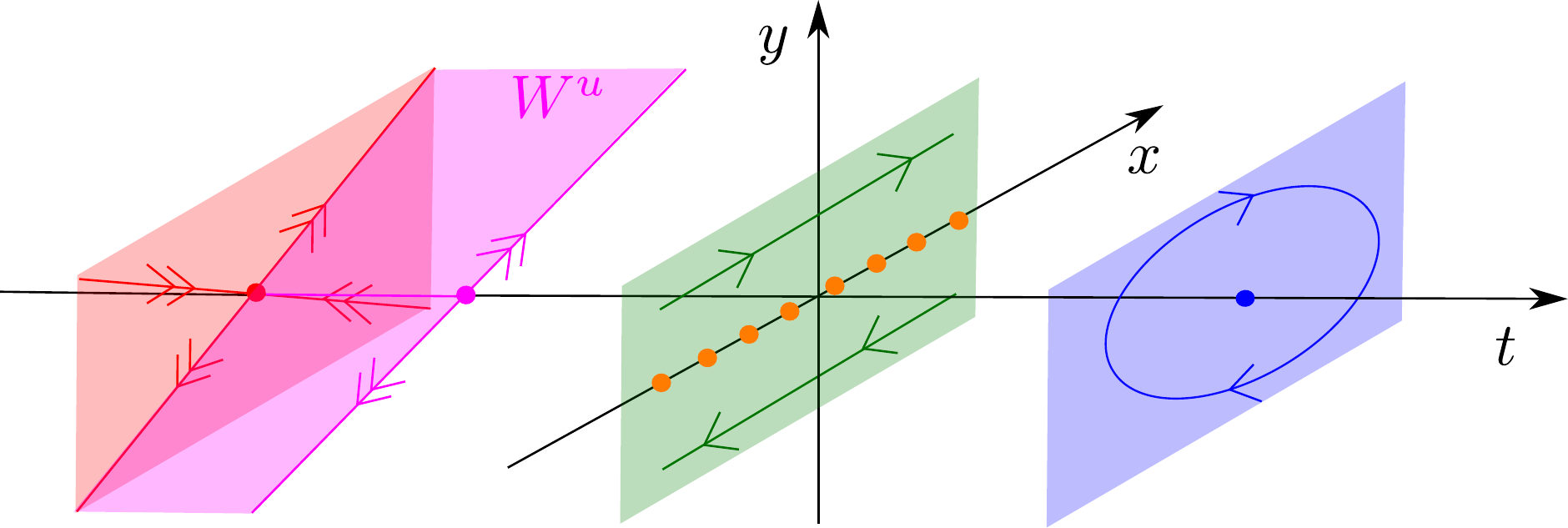}}
\end{center}
 \caption{Limiting $\epsilon =0$ structure of a simple linear turning point at $t=0$, dividing the critical manifold along the $t$-axis into hyperbolic parts (for $t<0$ locally, shown in red) and elliptic parts (for $t>0$ locally, shown in blue). For $t=0$, the $x$-axis (shown in orange) is a line of degenerate equilibria (since the linearization is nilpotent). The  unstable manifold $W^u$ of the critical manifold  for $t < \delta <0$ is shown in pink.
\figlab{xyt}}
\end{figure}
 %
In particular, Fenichel Theory \cite{fen3,jones_1995,kuehn2015}  implies that for $\mu \in C^k$, $k\ge 2$,  these unstable and  stable manifolds perturb to unstable and stable manifolds  $W^u_{\epsilon}$ and $W^s_{\epsilon}$, respectively,
of the (trivial) slow manifold $(0,0,t)$, $t \leq -\delta < 0$ of class $C^k$  (including the $\epsilon$ dependence)  for $\epsilon$ sufficiently small.
Due to the linearity of the problem, the perturbed manifolds $W^u_{\epsilon}$ and $W^s_{\epsilon}$ are in fact line bundles. 
The exponentially growing and decaying WKB-solutions are asymptotic expansions of these
 unstable and stable manifolds, respectively, see \secref{mainResults}. 
 
 For $t>0$ the situation is different. Here standard GSPT breaks down. To deal with fast oscillations, one typically applies averaging \cite{geller1,Guckenheimer97,neishtadt1987a}, but in the present context of linear problems, see also \cite{wasow1965a}, it is more natural to look for a diagonalization. In \secref{mainResults}, we will show, 
  see \lemmaref{diagsimple}, that the diagonalization for $t\ge \delta>0$, in the case where $\mu$ is analytic, is a consequence of existence of normally elliptic slow manifolds \cite{dema}. (\lemmaref{uvN} deals with $t\ge \delta>0$ in the finitely smooth case; these Lemmas provide  new proofs (to the best of our knowledge) for the validity of the Liouville-Green (WKB)-approximation.)
%

Inspired by the success of blowup in nonlinear problems with hyperbolic-to-hyperbolic transitions through nonhyperbolic sets, see e.g. \cite{Gucwa2009783,krupa_extending_2001}, we will in this paper cover a full neighborhood of $t=0$ for all $0<\epsilon\ll 1$, and describe the transition of $W^u$, by blowing up the degenerate set $(y,t,\epsilon)=(0,0,0)$, $x\in \mathbb R$, (orange in \figref{xyt}) of \eqref{system1}, to a cylinder of spheres. Through blowup and desingularization, we essentially amplify the vanishing eigenvalues to nontrivial ones. In the present case, due to the hyperbolic-to-elliptic transition, we obtain real eigenvalues on one side of the cylinder and imaginary ones on the other side. In line with \cite{dumortier1996a,krupa_extending_2001}, the real eigenvalues allow us to extend the hyperbolic spaces $W_\epsilon^{s,u}$, obtained by Fenichel's theory and GSPT, to $t=- c\epsilon^{2/3}$ for $c>0$ large enough and all $0<\epsilon\ll 1$, through the use of center manifold theory \cite{car1}. In this paper, we will show that the diagonalization procedure, used within $t\ge \delta>0$ -- which related to existence of normally elliptic slow manifolds in the analytic setting ( \lemmaref{diagsimple})-- can also be extended to $t=c\epsilon^{2/3}$ through what we refer to as ``elliptic center manifolds''; these are invariant graphs over the zero eigenspace in the presence of imaginary eigenvalues. We will obviously explain this more carefully later on, but the existence of such manifolds, which require analyticity, is perhaps less known in the dynamical systems community.\footnote{This is the only place where we use (in a minimal way) Gevrey properties and Borel/Laplace techniques (following \cite{bonckaert2008a}), that -- along with formal series -- have been central to many other approaches to the turning point problem (see further discussion of this in \secref{other}).}
The remaining gap from $t=-c\epsilon^{2/3}$ to $t=c\epsilon^{2/3}$ is covered by the scaling \eqref{airyscaling} (which relates to the scaling chart associated to the blowup transformation) and the solutions of Airy \eqref{airy0}. In this way, we obtain a rigorous  matching across the turning point.



We feel that our approach, based upon dynamical systems theory and blowup, sheds some new light on the WKB-method and the associated notoriously difficult turning point problems.
The novelty of our approach partially lies in the fact that we work with well defined dynamical objects, i.e. invariant manifolds. Asymptotic expansions appear at a later
stage as approximations of these  geometric objects. Thus the focus is on the geometry and dynamics,  which can be studied 
by well developed methods,  i.e. center manifold, slow manifolds, Fenichel theory,  blow-up, and normal form transformations.  
This leads to understanding of the dynamics/solutions  rather than  merely computing expansions.
As a further novelty, our approach works under finite smoothness requirements. In this regard, it is important to highlight that although the ``elliptic center manifolds'' -- which are central to our approach -- require analyticity, they will only be needed (in a normal form procedure) along certain polynomial expansions that appear naturally in the blowup procedure. \response{We also use our approach to show that the unstable manifold on the elliptic side is a smooth function of $\epsilon^{1/3}$ (in a certain sense which we make precise below). We believe that this result is new and interesting. In particular, this feature is not visible in a formal WKB-approach where exponentially growing solutions are matched to oscillatory solutions through the Airy function. }

\response{
Interestingly, the dynamics along critical manifolds going from nodal to focus normal stability are also related to hyperbolic-elliptic transitions (upon going to exponential weights). See \cite{carter2018a}. In this paper, the authors also use (a different) blowup to describe transitions near such points (called Airy points in \cite{carter2018a}), in the context of pulse transitions in the FitzHugh-Nagumo system, and it is found that Airy-functions play an important role. However, the full hyperbolic-elliptic transition is not covered in \cite{carter2018a} (as it is not important for the pulse transitions) and consequently the result of this paper does not cover our case. The paper \cite{carter2021a} studies a general class of eigenvalue problems that does not cover, but resembles, the stability problem of the pulse solutions of \cite{carter2018a}. It is found that the presence of Airy points lead to a certain accumulation of eigenvalues in the singular limit. Such accumulation also occurs for the Schr{\"o}dinger eigenvalue problem, with eigenvalues separated by $\mathcal O(\epsilon)$-distances as $\epsilon\rightarrow 0$, see \eqref{quant}.  However, the results of \cite{carter2021a} are qualitative and not related to our objective of providing detailed description of the transition near turning points of \eqref{secondordereq}.}
\subsection{Other approaches to turning point problems}\seclab{other} 
To put our work into further perspective, we give a brief overview on other approaches to WKB-type problems and 
to the corresponding  turning point problems.
In  vector-matrix notation these problems have the form
\begin{equation}\eqlab{system2}
\begin{aligned}
 \dot u &=A(t,\epsilon) u,\\
  \dot t&=\epsilon.
\end{aligned}
\end{equation}
Often a basic assumption is that $A(t,0)$ can be diagonalized or block-diagonalized.
The basic question is then whether the system \eqref{system2} can be diagonalized  or block-diagonalized
for $0 < \epsilon \ll 1$ by a suitable transformation $u = P(t,\epsilon) v$.
Exceptional points $t_0$ where this is not possible in a full neighborhood of $t_0$ are called turning points \cite{wasow1965a,wasow1985a}.
As in system \eqref{system1} turning points are often related to degeneracies of the spectrum of $A(t,0)$ lying on or collapsing onto
 the imaginary axis at $t=t_0$. In the language of GSPT this is often associated with a loss of normal hyperbolicity in one way or another. 
Note, however, that separated purely imaginary eigenvalues of $A(t,0)$ do not cause problems from the WKB-point of view. (Obviously, formal WKB-expansions may not correspond to true solutions, see e.g. \cite{kruskal1991a} for a related problem).
Very powerful results have been obtained by treating such problems in an analytic setting, i.e. by considering $\mu(t)$ and the solution $x(t)$ 
or $A(t,\epsilon)$ and the solution $u(t)$ as analytic functions of  the  variable $t \in \C$, see \cite{fedoryuk1993, wasow1965a} and the references therein.

In the context of \eqref{system1} with $\mu$ analytic, an interesting result is \cite[Theorem 6.5.-1]{wasow1965a}, see also start of section \cite[Section 8.6]{wasow1965a}, showing that there exists a transformation, analytic in $t$ and with technical asymptotic (Gevrey) properties as $\epsilon\rightarrow 0$, which locally brings \eqref{system1} into the singularly perturbed normal form:
\begin{equation}\eqlab{xytairy0}
\begin{aligned}
 \dot x &=y,\\
 \dot y&=-t x,\\
 \dot t&=\epsilon.
\end{aligned}
\end{equation}
This system is equivalent to the scaled version of Airy \eqref{airy0} with $t = \epsilon^{2/3}\tau$: $$\epsilon^2 x''(t) = -tx(t).$$ Since the solutions of the Airy-equation are known, in this way one can in the analytic setting (in principle) track the unstable manifold across $t=0$ (see details in \eqref{airytrack} below). 

Going much beyond this classical work on the asymptotic analysis of systems \eqref{system1} and \eqref{system2}, an impressive arsenal 
of powerful methods has been developed since the 1990s  known as ``exact WKB methods'' pioneered by A. Voros \cite{voros1983a}  and resummation methods based on the concept of 
``resurgence'' introduced by \'Ecalle \cite{ecalle,mitschi2016a}. The common feature of these approaches is that they allow to give a meaning to the 
divergent formal asymptotic expansions and to view them as encodings of genuine solutions.

\subsection{Overview}
In \secref{mainResults}, we first present our dynamical systems approach to WKB on either side of $t=0$. 
Subsequently, in \secref{tpmain}, we present our main results on the connection problem for the turning point, see \thmref{mainTP} and \thmref{mainTP2}. The proofs of these statements then follow in \secref{proofTP}. In \secref{discussion}, we conclude the paper with a discussion section that also focuses on future work. 

\section{A dynamical systems approach to WKB}\seclab{mainResults}
We consider \eqref{system1} 
and assume (for the most part) that $\mu$ is $C^{k}$-smooth with $k\ge 5$ and satifies
 \begin{align}
  \mu(0)=0,\,\mu'(0)>0.\eqlab{mucond}
 \end{align}
Specifically, by \eqref{mucond} it follows that there is a neighborhood $\mathcal N$ of $0$, such that in $\mathcal N$ we have $\mu(t)=0\Leftrightarrow t=0$. Henceforth, we will only work locally in $\mathcal N$. 


Let $\lambda(t)=\sqrt{-\mu(t)}$. 
%
%
Then for $\epsilon=0$, any point $(0,0,t)$ is an equilibrium of \eqref{system1} and the linearization has $\pm \lambda(t),0$ as eigenvalues. Therefore $(0,0,t)$ is partially hyperbolic for $t<0$, where $\lambda(t)\in \mathbb R$, and elliptic for $t>0$, where $\lambda(t)\in i\mathbb R$.  Consequently, for $t<0$ (within $\mathcal N$) we have stable and unstable manifolds $W^s$ and $W^u$. In the following we describe the perturbation of these manifolds. 
\begin{lemma}\lemmalab{WutNeg}
Suppose that $\mu \in C^{k}$ with $k\ge 2$ and let $I_-\subset \mathcal N$ be a compact interval contained within $t<0$. Then the stable and unstable manifold for $t\in I_-$ perturb to $W^s_\epsilon$ and $W^u_\epsilon$ for all $0<\epsilon\ll 1$. Specifically, $W_\epsilon^u$ is a line bundle taking the following graph form
\begin{align}\eqlab{Wu}
 W_\epsilon^u:\quad 
y = h_u(t,\epsilon) x,\quad t\in I_-,\,x\in \mathbb R.
\end{align}
Here $h_u$ is a smooth function; in fact, a simple computation shows that
\begin{align}
 h_u(t,\epsilon) = \sqrt{-\mu(t)} -\frac14 \mu(t)^{-1} \mu'(t) \epsilon + \epsilon^2 U_2(t,\epsilon),\eqlab{hufunc}
\end{align}
with $U_2\in C^{k-2}$.
\end{lemma}
\begin{proof}
 The perturbation of $W^{s}$ and $W^u$ follow from Fenichel's theory and GSPT. The form \eqref{Wu} with expansion \eqref{hufunc} can easily be obtained by using projective coordinates $u=x^{-1} y$ so that 
 \begin{equation}\eqlab{utsyst}
 \begin{aligned}
  \dot u &= -\mu(t)-u^2,\\
  \dot t&=\epsilon.
 \end{aligned}
 \end{equation}
$W^u$ then corresponds to a slow manifold $u=h_{u}(t,\epsilon)$ obtained as perturbations of the hyperbolic and attracting critical manifold $u=\sqrt{-\mu(t)}$ of \eqref{utsyst} with $t\in I_-$ for $\epsilon=0$. The slow manifold can easily be approximated by expanding $u=h_u(t,\epsilon)$ in $\epsilon$ and collecting terms. This produces \eqref{hufunc}. \end{proof}

Henceforth we write $W^{s/u}_\epsilon$ as $W^{s/u}$; we believe it is clear from the context whether $\epsilon=0$ or $0<\epsilon\ll 1$. These manifolds are not unique, but are exponentially close with respect to $\epsilon\rightarrow 0$, i.e. different choices alter by $\mathcal O(e^{-c/\epsilon})$-distances. This lack of uniqueness does not play a role in the statements that follow and we shall henceforth keep $W^u$ fixed.

Using the fact that $W^u$ is invariant, we can compute solutions by substituting $y=h_u(t,\epsilon)$ into the equation for $x$. This gives
\begin{align}
 \frac{dx}{dt} = \epsilon^{-1} h_u(t,\epsilon) x,\eqlab{xWu}
\end{align}
which can be integrated.
\begin{proposition}
Consider a point $q_0(\epsilon):(x_0,y_0,t_0)\in W^u$ with $t_0\in I_-$ and $0<\epsilon\ll 1$. Then $y_0=h_u(t_0,\epsilon)x_0$ and we can write the solution $(x(t),y(t),t)$, $t\in I_-$, through $q_0(\epsilon)$ as 
\begin{align*}
 x(t)&= \sqrt[4]{\frac{\mu(t_0)}{\mu(t)}}\exp \left(\epsilon^{-1} \int_{t_0}^t \sqrt{-\mu(s)} ds\right)\left(1+\mathcal O(\epsilon)\right)x_0,\\
 y(t)&= \sqrt[4]{\mu(t)\mu(t_0)}\exp \left(\epsilon^{-1} \int_{t_0}^t \sqrt{-\mu(s)} ds\right)\left(1+\mathcal O(\epsilon)\right)x_0.
\end{align*}
\end{proposition}
\begin{proof}
We simply solve \eqref{xWu} using the expansion of $h_u$ in \eqref{hufunc}:
\begin{align*}
 x(t)x_0^{-1} &=\exp\left(\epsilon^{-1} \int_{t_0}^t h_u(\tau,\epsilon) d\tau\right) \\
 &=\exp\left(-\frac14 \int_{t_0}^t \mu(\tau)^{-1} \mu'(\tau) d\tau\right) \exp\left(\int_{t_0}^t \sqrt{-\mu(\tau)} d\tau\right)\exp\left(\epsilon \int_{t_0}^t U_2(\tau,\epsilon) d\tau\right)\\
 &=\exp\left(-\frac14 \int_{t_0}^t \mu(\tau)^{-1} \mu'(\tau) d\tau\right) \exp\left(\int_{t_0}^t \sqrt{-\mu(\tau)} d\tau\right)(1+\mathcal O(\epsilon)),
\end{align*}
which gives the result.

\end{proof}

On the elliptic side $t>0$, things are more complicated from the dynamical systems point of view.
      In the following result, we restrict to $\mu$ analytic. For the statement, we need a definition of Gevrey-1 functions: Let $S_\theta\subset \mathbb C$ denote the sector 
\begin{align}\eqlab{Stheta}
S_\theta = \{z\in \mathbb C:0\le \vert \arg{z}\vert <\theta/2\},\end{align}
centered along the positive real axis with opening $\theta\in(0,2\pi)$. We then recall that an analytic function $h:S_\theta \rightarrow \mathbb C$, having a continuous extension to $z=0$ with $h(0)=0$, is said to be Gevrey-1 with respect to $z\in S_\theta$ if there are $\{h_n\}_{n\in \mathbb N}$, $a>0$ and $b>0$ such that 
\begin{align}
 \vert h(z)-\sum_{n=1}^{N-1} h_n z^n \vert \le a b^n n!\vert z\vert^N,\eqlab{gevrey1}
\end{align}
for all $N$, \cite{balser1994a}. In further generality, it is possible to define Gevrey-1 functions within sectors centered along different directions, but we will not need this here. We will also consider functions $h(t,z)$ that are analytic in $t\in D$, $z\in S_\theta$, in particular Gevrey-1 with respect to $z$ uniformly in $t\in D$. By this we mean that $h_n$ are analytic functions of $t\in D$ in \eqref{gevrey1}, whereas the constants $a$ and $b$ can be taken to be independent, see also \cite{balser1994a,dema}. We will for simplicity often suppress $S_\theta$ in the following; for us the important thing is that it is centered along the positive real axis. 
  \begin{lemma}\lemmalab{diagsimple}
Let $I_+\subset \mathcal N$ be a compact interval within $t>0$, recall that $\lambda(t)=i\sqrt{\mu(t)}$ in this case, and suppose that $\mu$ is a real-analytic function on $\mathcal N$. 
Then there exists an $\epsilon_0>0$ such that for all $0\le \epsilon\le \epsilon_0$ the following holds: There exists a transformation
\begin{align}
 (u,v,t)\mapsto (x,y),\eqlab{uvtxy}
 \end{align}
 which is analytic in $t\in I_+$, Gevrey-1 with respect to $\epsilon$ uniformly in $t\in I_+$, and a linear isomorphism in $u,v$ for fixed $t\in I_+$, such that \eqref{system1} becomes
 \begin{equation}
\begin{aligned}
 \dot u &=\nu(t,\epsilon) u,\\
 \dot v&=\overline \nu(t,\epsilon) v,\\
 \dot t&=\epsilon,
\end{aligned}\eqlab{uvdiag0}
\end{equation}
with 
\begin{align}\eqlab{nuexpr0}
\nu(t,\epsilon)=\lambda(t) f(t,\epsilon) -\epsilon \lambda(t)^{-1} \lambda'(t),\end{align} also being analytic in $t$ and uniformly Gevrey-1 with respect to $\epsilon$, satisfying 
\begin{align}
\nu(t,\epsilon) &= \lambda(t) -\frac12 \lambda(t)^{-1} \lambda'(t)\epsilon +\epsilon^2 T_2(t,\epsilon),\eqlab{nupm3}
\end{align}
for some $T_2$ (with identical analytic properties). 
\end{lemma}
\begin{proof}
 We look for the desired transformation in the following form:
 \begin{align}\eqlab{xyuv}
  \begin{pmatrix}
   x\\
   y
  \end{pmatrix}&=
\begin{pmatrix}
   f(t) & \overline f(t)\\
   \lambda(t) & -\lambda(t)
  \end{pmatrix}\begin{pmatrix}
  u\\
  v
  \end{pmatrix}.
 \end{align}
  $f$ will also depend upon $\epsilon$ but we will only emphasize this (by writing $f(t,\epsilon)$) when necessary. 
Inserting \eqref{xyuv} into \eqref{system1} gives
\begin{align*}
 ({f (t) +\overline f(t) }) u' &=\left(
\lambda (t)(1+\vert f(t) \vert^2) -
\epsilon f'(t)-\epsilon \lambda(t)^{-1} 
\lambda'(t)\overline f(t) 
\right) u \\
&+ \left(  -\lambda (t)(1- \overline f(t)^{2}) -
 \epsilon \overline f'(t) +\epsilon  \lambda(t)^{-1}\lambda'(t)  \overline f(t) 
\right) v,\\
({f (t) +\overline f(t) }) v' &= \left( \lambda (t)(1-  f(t)^{2})
  (t)-
 \epsilon f'(t) +\epsilon  \lambda(t)^{-1}\lambda'(t)  f(t) 
\right) u\\
&+\left(
-\lambda (t)(1+\vert f(t) \vert^2) -
\epsilon \overline f'(t)-\epsilon \lambda(t)^{-1} 
\lambda'(t)f(t) 
\right)v,
\end{align*}
and hence the desired diagonalization provided the off-diagonal terms vanishes. Since these are complex conjugated (using $\overline \lambda=-\lambda$ in this case), 
this condition reduces to the following singularly perturbed ODE for $f$
\begin{align}
 \epsilon f'(t) &=\lambda(t)(1-f(t)^2)+\epsilon \lambda(t)^{-1} \lambda'(t) f(t),\eqlab{ft}
\end{align}
with $0<\epsilon\ll 1$,
or equivalently as a first order system:
\begin{equation}\eqlab{ft2}
\begin{aligned}
 \dot f &=\lambda(t)(1-f^2)+\epsilon \lambda(t)^{-1} \lambda'(t) f,\\
 \dot t &=\epsilon.
\end{aligned}
\end{equation}
Seeing that $\lambda(t)$ is purely imaginary, it follows that the set defined by $f=1$ is a normally elliptic critical manifold for \eqref{ft2} for $\epsilon=0$. By the results in \cite{canalis-durand2000a,de2020a}, there exists an invariant manifold of \eqref{ft}, given by $f(t,\epsilon)$ with $f(t,0)=1$ being analytic in $t$ and Gevrey-1 in $\epsilon$. Given such $f(t,\epsilon)$, a simple calculation shows that \eqref{xyuv} produces the following equations:
\begin{equation}\eqlab{uvdiag}
\begin{aligned}
 \dot u &=\left(\lambda(t) f(t,\epsilon) -\epsilon \lambda(t)^{-1} \lambda'(t)\right) u,\\
 \dot v &=\left(-\lambda(t) \overline f(t,\epsilon) -\epsilon \lambda(t)^{-1} \lambda'(t)\right) v,
\end{aligned}
\end{equation}
so that $\nu(t,\epsilon) =\lambda(t) f(t,\epsilon) -\epsilon \lambda(t)^{-1} \lambda'(t)$, seeing that $\lambda$ is purely imaginary. A simple calculation shows that 
\begin{align}
 f(t,\epsilon) &=1+\frac12 \lambda(t)^{-2} \lambda'(t) \epsilon +\mathcal O(\epsilon^2),\eqlab{ffex}
\end{align}
which upon inserted into the expression for $\nu$ gives \eqref{nupm3}.
\end{proof}
\begin{remark}
 \response{Consider a formal power series expansions  $\sum_{m=0}^\infty \tilde \nu_m(t) \epsilon^m$ of $\nu$ and let $\nu_{\text{even}}$ and $\nu_{\text{odd}}$ be the splittings into even and odd powers:
 \begin{align}
  \nu_{\text{even}}(t) = \sum_{n=0}^\infty \tilde \nu_{2n}(t)\epsilon^{2n},\quad \nu_{\text{odd}}(t) = \sum_{n=0}^\infty \tilde \nu_{2n+1}(t)\epsilon^{2n+1}.\nonumber 
 \end{align}
Then it can be shown using \eqref{ft} that
\begin{align}
  \int \nu_{odd}(t) dt =-\frac{\epsilon}{2} \log \vert\nu_{even}(t)\vert,\eqlab{evenodd2}
\end{align}
with equality understood in terms of formal series, (i.e. term wise).
A similar property holds at the level of the WKB-expansion, see e.g. \cite[Chapter 11]{quantum}. The consequence is that upon integrating the $u$-equation in \eqref{uvdiag0} we have (at the formal level)
\begin{align*}
 u(t) = \sqrt{\frac{\nu_{\text{even}}(t)}{\nu_{\text{even}}(t_0)}} e^{\epsilon^{-1} \int_{t_0}^t \nu_{\text{even}}(s) ds}u(t_0).
\end{align*}}
\end{remark}

Following \lemmaref{diagsimple} we can (in the analytic case) describe solutions for $t\in I_+$. Indeed, we can just integrate \eqref{uvdiag}\response{
\begin{align}
 u(t) &= \exp\left(\int_{t_0}^t \nu(s,\epsilon)ds\right)u(t_0)\nonumber\\
 &=\sqrt[4]{\frac{\mu(t_0)}{\mu(t)}}\exp\left(\epsilon^{-1} i \int_{t_0}^t \sqrt{\mu(s)}ds\right)\exp\left(\epsilon \int_{t_0}^t T_2(s,\epsilon)ds \right)u(t_0),\eqlab{uexpansion0}
\end{align}
$v(t)=\overline u(t)$ using \eqref{nupm3},
and transform the result back to $x$ and $y$ using 
\begin{align*}
x(t)&=2\operatorname{Re}(f(t,\epsilon)u(t))\\
y(t)&=-2 \sqrt{\mu(t)}\operatorname{Im}(u(t)).
\end{align*}}
To leading order, we obtain the Liouville-Green (WKB)-  approximation \cite{wasow1985a}:
 \begin{equation}
\begin{aligned}
x(t)&=\sqrt[4]{\frac{\mu(t_0)}{\mu(t)}} \bigg[\left(\cos \left(\epsilon^{-1}\int_{t_0}^t \sqrt{\mu(\tau)}d\tau\right)+\mathcal O(\epsilon)\right)x(t_0)\\
&+\frac{1}{\sqrt{\mu(t_0)}} \left(\sin \left(\epsilon^{-1}\int_{t_0}^t \sqrt{\mu(\tau)} d\tau\right)+\mathcal O(\epsilon)\right)y(t_0)\bigg],\\
y(t)&=\sqrt[4]{\frac{\mu(t_0)}{\mu(t)}} \bigg[-\sqrt{\mu(t)}\left(\sin \left(\epsilon^{-1}\int_{t_0}^t \sqrt{\mu(\tau)} d\tau\right)+\mathcal O(\epsilon)\right)x(t_0)\\
&+\sqrt{\frac{\mu(t)}{\mu(t_0)}} \left(\cos \left(\epsilon^{-1}\int_{t_0}^t \sqrt{\mu(\tau)}d\tau\right)+\mathcal O(\epsilon)\right)y(t_0)\bigg],
  \end{aligned}\eqlab{xytxyt02}
  \end{equation}
  \response{The remainder terms in \eqref{xytxyt02} are only $C^0$. The regularity with respect to $\epsilon$ is better represented in the form:
   \begin{equation}
\begin{aligned}
x(t)&=2\sqrt[4]{\frac{\mu(t_0)}{\mu(t)}} \left(1+\epsilon \rho_1(t,\epsilon)\right) \operatorname{Re}\left(e^{\epsilon^{-1} i \int_{t_0}^t \sqrt{\mu(s)}ds+\epsilon i\phi_1(t,\epsilon)}u(t_0)\right),\\
y(t)&=-2\sqrt[4]{\mu(t){\mu(t_0)}}\left(1+\epsilon \rho_2(t,\epsilon)\right)\operatorname{Im}\left(e^{\epsilon^{-1} i \int_{t_0}^t \sqrt{\mu(s)}ds+\epsilon i\phi_2(t,\epsilon)}u(t_0)\right)
  \end{aligned}\eqlab{xytxyt03}
  \end{equation}
  using \eqref{xyuv}, $\lambda(t)=i\sqrt{\mu(t)}$ and \eqref{uexpansion0}. Here we have for simplicity  kept $u(t_0)$ as the initial condition. It is a direct consequence of \lemmaref{diagsimple}, that the functions $\rho_i,\phi_i$, $i=1,2$, are both real analytic in $t\in I_+$ and uniformly Gevrey-1 with respect to $\epsilon$. 
  }
  
The results and the techniques in \cite{canalis-durand2000a,de2020a} rely heavily on the analyticity of $\mu$. Consequently, there is no reason to expect that \lemmaref{diagsimple} also holds in the smooth case. However, it is possible to obtain an asymptotic or ``quasi-diagonalization'' version of \lemmaref{diagsimple} in the smooth setting. We collect the result into the following lemma.
\begin{lemma}\lemmalab{uvN}
Fix any $N$ such that $\mu \in C^{k}$ with $k\ge N+1$. Then there exists a $f_N$, being $C^{k-N}$ in $t\in I_+$ and polynomial of degree $N$ in $\epsilon$, such that the transformation $(u,v,t)\mapsto (x,y)$ defined by
  \begin{align}\eqlab{xyuvN}
  \begin{pmatrix}
   x\\
   y
  \end{pmatrix}&=
\begin{pmatrix}
   f_N(t,\epsilon) & \overline f_N(t,\epsilon)\\
   \lambda(t) & -\lambda(t)
  \end{pmatrix}\begin{pmatrix}
  u\\
  v
  \end{pmatrix},
 \end{align}
brings \eqref{system1} into the following form
 \begin{equation}\eqlab{uvtN}
\begin{aligned}
 \dot u &=\nu_N(t,\epsilon) u+\mathcal O(\epsilon^{N+1})v,\\
 \dot v&=\mathcal O(\epsilon^{N+1})v+\overline{\nu}_N(t,\epsilon)v,\\
 \dot t&=\epsilon.
\end{aligned}
\end{equation}
Here $\nu_N(t,\epsilon)=\lambda(t) f_N(t,\epsilon) -\epsilon \lambda(t)^{-1} \lambda'(t)+\mathcal O(\epsilon^{N+1})$ is $C^{k-N-1}$ in $t$, analytic in $\epsilon$ and satisfies
\begin{align}
\nu_N(t,\epsilon) &= \lambda(t) -\frac12 \lambda(t)^{-1} \lambda'(t)\epsilon +\epsilon^2 T_{2,N}(t,\epsilon),\eqlab{nupm}
\end{align}
for some $T_{2,N}$ having the same properties.
\end{lemma}

\begin{proof}
First, we suppose that $k=\infty$. We will then deal with the finite smoothness towards the end of the proof. 

Instead of solving \eqref{ft2} exactly, we look for ``quasi-solutions'' defined in the following sense: Write the equation \eqref{ft2} as $F(f,t,\epsilon)=0$. Then $f_N(t,\epsilon)$ smooth is a ``quasi-solution'' of order $\mathcal O(\epsilon^{N+1})$ if $F(f_N(t,\epsilon),t,\epsilon)=\mathcal O(\epsilon^{N+1})$ for $N\in \mathbb N$. It is standard that such quasi-solutions can be obtained as Taylor-polynomials $f_N(t,\epsilon)=\sum_{n=0}^{N} R_n(t)\epsilon^n$, see e.g. \cite{de2020a,kristiansenwulff}, with $R_n$ recursively starting from $R_0(t)=1$ in the present case. In fact, a simple calculation shows that $R_n$ is given by
\begin{align}
 R_n &=-\frac12 \sum_{l=1}^{n-1} R_l R_{n-l}+\frac12 \lambda^{-1} R_{n-1}'+\frac12 \lambda^{-2} \lambda' R_{n-1},\eqlab{Rneqn}
\end{align}
for $n\ge 1$. 
Consequently, we find that for each fixed $N\in \mathbb N$ there is transformation (which is polynomial in $\epsilon$) so that \eqref{uvdiag} holds up to off-diagonal remainder terms of order $\mathcal O(\epsilon^{N+1})$. Moreover, from \eqref{Rneqn} we find by induction on $n$ that $R_n\in C^{k-n}$ and consequently, $f_N\in C^{k-N}$ whenever $\mu\in C^k$. We lose one degree of smoothness of $\nu_N$ due to the $f'$-term in $F(f,t,\epsilon)$ and hence $\nu_N\in C^{k-N-1}$ as claimed. 

\end{proof}

From this quasi-diagonalization it is possible to recover \eqref{xytxyt02}.
We collect this result in the following proposition.
\begin{proposition}
 The Liouville-Green approximation \eqref{xytxyt02} holds for $\mu \in C^2$ and $\mu(t)\ge c>0$. 
\end{proposition}
\begin{proof}
We use \lemmaref{uvN} with $N=1$.
This gives
\begin{align*}
 f_1(t,\epsilon) =1+R_1(t)\epsilon, 
\end{align*}
with $R_1(t)=\frac12 \lambda(t)^{-2} \lambda'(t)$,
see \eqref{Rneqn} with $n=1$. (Notice also that this agrees with the leading order expression in \eqref{ffex}). For this it is sufficient that $\mu \in C^2$. Consider
 \eqref{uvtN} and let
\begin{align*}
 Q(t) :=\exp\left(\int_{t_0}^t \frac{1}{\epsilon}\nu_1(\tau,\epsilon)d\tau\right),
\end{align*}
and define $\tilde u$ and $\tilde v$ by
\begin{align*}
 u = Q(t)\tilde u,\quad v = \overline Q(t)\tilde v.
\end{align*}
Then 
\begin{align*}
 \frac{d\tilde u}{dt} &=\mathcal O(\epsilon) Q(t)^{-1} \overline Q(t) \tilde v,\\
 \frac{d\tilde v}{dt} &=\mathcal O(\epsilon) \overline Q(t)^{-1} Q(t) \tilde u.
\end{align*}
Since $\vert \overline Q(t)^{-1} Q(t)\vert$ it is straightforward to integrate these equations from $t_0$ to $t$ and estimate
\begin{align*}
 \tilde u(t) &= (1+\mathcal O(\epsilon))\tilde u(t_0)+\mathcal O(\epsilon)\tilde v(t_0),\\
 \tilde v(t) &= (\mathcal O(\epsilon))\tilde u(t_0)+(1+\mathcal O(\epsilon))\tilde v(t_0).
\end{align*}
Upon returning to $x,y$ we recover the Liouville-Green approximation \eqref{xytxyt02} as claimed.
\end{proof}
 For higher order improvements of \eqref{xytxyt02}, including $C^k$-versions of \eqref{xytxyt03}, $N>1$  and $C^{k}$ with $k>2$ are necessary. For simplicity, we leave out such results from the present manuscript.

\section{The turning point: main results}\seclab{tpmain}
Having studied solutions on either side of the turning point within $t\in I_-$ and $t\in I_+$, we now turn our attention to the main problem: The description of solutions within $W^u$ across the turning point. 

In particular, in terms of solving eigenvalue problems of the form \eqref{eq0}, we are interested in describing the unstable manifold $W^u$ for $t>0$ fixed for all $0<\epsilon\ll 1$. The hyperbolic theory describes this space on the $t<0$ side, recall \lemmaref{WutNeg}, but the theory offers no control over this object as $t$ crosses $0$. 

Before presenting our main results, consider the case $\mu(t)=t$. Then inserting 
\begin{equation}\eqlab{scaling2}
\begin{aligned}
y&=\epsilon^{1/3} y_2,\\
t &=\epsilon^{2/3} t_2,
\end{aligned}
\end{equation}
into \eqref{system1} gives
\begin{equation}\eqlab{airy2}
\begin{aligned}
\dot x&=y_2,\\
\dot y_2 &=-t_2 x,\\
\dot t_2 &=1,
 \end{aligned}
 \end{equation}
 upon dividing the right hand side by $\epsilon^{1/3}$. This system is obviously equivalent to the Airy equation:
\begin{align}
 x''(t_2) =-t_2 x(t_2),\eqlab{airyeqn}
\end{align}
with two linearly independent solutions $\operatorname{Ai}(-\cdot)$ and $\text{Bi}(-\cdot)$, the former being the Airy-function. The $\operatorname{Ai}(-\cdot)$-solution has the following asymptotics:
\begin{align}
 \operatorname{Ai}(-t_2) &= \frac{1}{2\pi^{1/2} \vert t_2\vert^{1/4} }e^{-2 \vert t_2\vert^{3/2}/3}(1+\mathcal O(\vert t_2\vert^{-3/2})) \quad \mbox{for}\quad t_2\rightarrow -\infty,\eqlab{Aitau}\\
 \operatorname{Ai}(-t_2) &= \frac{1}{\pi^{1/2} t_2^{1/4} }\cos\left(\frac{2}{3}t_2^{3/2}-\frac{\pi}{4}+\mathcal O(t_2^{-3/2})\right)(1+\mathcal O(t_2^{-3/2}))
 \quad \mbox{for}\quad t_2\rightarrow \infty.\eqlab{Aitau_}
\end{align}
In the case $t_2\rightarrow \infty$, we also have that 
\begin{align}
 \operatorname{Bi}(-t_2) &= -\frac{1}{\pi^{1/2} t_2^{1/4} }\sin\left(\frac{2}{3}t_2^{3/2}-\frac{\pi}{4}+\mathcal O(t_2^{-3/2})\right)(1+\mathcal O(t_2^{-3/2})),\eqlab{Bitau_}\\
\operatorname{Ai}'(-t_2) &= \frac{t_2^{1/4}}{\pi^{1/2} }\sin\left(\frac{2}{3}t_2^{3/2}-\frac{\pi}{4}+\mathcal O(t_2^{-3/2})\right)(1+\mathcal O(t_2^{-3/2})),\eqlab{AiPtau_}\\
\operatorname{Bi}'(-t_2) &= \frac{t_2^{1/4}}{\pi^{1/2} }\cos\left(\frac{2}{3}t_2^{3/2}-\frac{\pi}{4}+\mathcal O(t_2^{-3/2})\right)(1+\mathcal O(t_2^{-3/2})),\eqlab{BiPtau_}
\end{align}
see \cite{AbramowitzStegun1964}.
The exponential decay of $x(t_2)=\operatorname{Ai}(-t_2)$ for $t_2\rightarrow -\infty$ shows that 
\begin{equation}\eqlab{airytrack}
\begin{aligned}
x &= \operatorname{Ai}(-\epsilon^{-2/3}t),\\
y &=-\epsilon^{1/3}\operatorname{Ai}'(-\epsilon^{-2/3}t),
\end{aligned}
\end{equation}
will provide the desired tracking of the unstable manifold $W^u$ for the simple case \eqref{xytairy0} with $\mu(t)=t$. 

For the general system \eqref{system1} we suppose that $\mu'(0)=1$ without loss of generality and define
\begin{align}
\widehat \mu(t):= \begin{cases} t^{-1} \mu(t), &t\ne 0\\
                   1 &t = 0
                  \end{cases}.
\eqlab{muhat}
\end{align}
Notice that if $\mu\in C^k$ then $\widehat \mu \in C^{k-1}$. 

In the following, we write the complex valued function defined by
$\operatorname{Ai}(t)-i\operatorname{Bi}(t)$
as $$\operatorname{Ai}-i\operatorname{Bi}.$$ 

We now state our main results \response{in terms of two theorems, that describe the  tracking of the unstable manifold $W^u$ of \eqref{system1} across the turning point $t=0$. Firstly in \thmref{mainTP}, we state an expansion of a solution within the unstable manifold, that is uniformly valid across the turning point. We state this result with (what we find to be) the least required smoothness of $\mu$.  }


\begin{theorem}\thmlab{mainTP}
Fix $\delta>0$, $\nu>0$ both small enough and suppose that $\mu\in C^{5}$ with $\mu(0)=0$, $\mu'(0)=1$. Next, consider $t\in I:=[-\nu,\nu]\subset \mathcal N$ and define the following intervals \begin{align}
  J_1&:=[-\nu,-\epsilon^{2/3} \delta^{-2/3}],\\
  J_2&:= [-\epsilon^{2/3} \delta^{-2/3},\epsilon^{2/3} \delta^{-2/3}],\\
  J_3&:=[\epsilon^{2/3} \delta^{-2/3},\nu],                                                                                                                                               \end{align}
such that $I=\cup_{i=1}^3 J_i$, while the open intervals $\textnormal{int}\,J_i, i=1,2,3,$ are disjoint for all $0<\epsilon\ll 1$. Then for any $0<\epsilon\ll 1$ the following holds:
\begin{enumerate}
 \item \label{expansionsol} There exists a solution $(x(t),y(t),t)\in W^u$, $t\in I$ of \eqref{system1}, having the following asymptotic expansions:
 \begin{enumerate}
\item \label{expansion1} The following holds uniformly within $t\in J_1$: \begin{align*}
   x(t)&=
 \widehat \mu(t)^{-1/4} \operatorname{Ai}(-\epsilon^{-2/3}t)e^{\frac{1}{\epsilon}\left(\int_{0}^t \sqrt{-\mu(t)}dt-\frac23 t^{3/2}\right)}(1+\mathcal O(\epsilon^{2/3}))\\
y(t)&=-\epsilon^{1/3} \widehat \mu (t)^{1/4} \operatorname{Ai}'(-\epsilon^{-2/3}t)e^{\frac{1}{\epsilon}\left(\int_{0}^t \sqrt{-\mu(t)}dt-\frac23 t^{3/2}\right)}(1+\mathcal O(\epsilon^{2/3})).
\end{align*}
\item \label{expansion2} The following holds uniformly within $t\in J_2$:  
\begin{align*}
   x(t)&=\operatorname{Ai}(-\epsilon^{-2/3}t)\left(1+\mathcal O(\epsilon^{2/3})\right)\\
y(t)&=-\epsilon^{1/3}\operatorname{Ai}'(-\epsilon^{-2/3}t)\left(1+\mathcal O(\epsilon^{2/3})\right).
\end{align*}
\item \label{expansion3} The following holds uniformly within $t\in J_3$:
\begin{align*}
   x(t)=&
 \widehat \mu(t)^{-1/4} \textnormal{Re}\bigg((\operatorname{Ai}-i\operatorname{Bi})(-\epsilon^{-2/3} t)e^{\frac{i}{\epsilon}\int_{0}^t \left(\int_{0}^t \sqrt{\mu(t)}dt-\frac23 t^{3/2}\right)}\times \\
 &(1+\mathcal O(\epsilon^{2/3}))\bigg)\\
y(t)=&-\epsilon^{1/3}\widehat \mu(t)^{1/4} \textnormal{Re}\bigg((\operatorname{Ai}-i\operatorname{Bi})'(-\epsilon^{-2/3} t)e^{\frac{i}{\epsilon}\int_{0}^t \left(\int_{0}^t \sqrt{\mu(t)}dt-\frac23 t^{3/2}\right)}\times \\
&(1+\mathcal O(\epsilon^{2/3}))\bigg).
\end{align*}
\end{enumerate}
\item \label{manifoldWu1} Moreover, let $W^u(\nu)$ denote the projection of $W^u\cap\{t=\nu\in \mathcal N\}$ with $\nu>0$ onto the $x,y$-plane. Then
\begin{align}
W^u(\nu) = \operatorname{span}\begin{pmatrix}
  \cos \left(\epsilon^{-1}\int_0^{\nu} \sqrt{\mu(\tau)}d\tau-\frac{\pi}{4}\right) + \mathcal O(\epsilon^{2/3})\\
  - \sqrt{\mu(\nu)}\sin \left(\epsilon^{-1}\int_0^{\nu} \sqrt{\mu(\tau)}d\tau-\frac{\pi}{4}\right) + \mathcal O(\epsilon^{2/3}) \end{pmatrix}.\eqlab{WuExpr}
%
\end{align}
\end{enumerate}
\end{theorem}
\response{In the next result, we present a detailed expansion of $W^u$ with $r_3=\nu>0$ fixed (small enough) in the $C^\infty$-case. With the view towards using the expansion to solve associated eigenvalue problems, we suppose that $\mu$ also depends smoothly on a parameter (the eigenvalue) $E$ in some appropriate domain $D\subset \mathbb R$. 
\begin{theorem}\thmlab{mainTP2}
Suppose that $\mu=\mu(t,E)$ is $C^\infty$, that $\mu(0,E)=0$, $\frac{\partial }{\partial t}\mu(0,E)>0$ for all $E\in D$, and consider any $M\in \mathbb N$. Let $W^u(\nu,E)$ denote the projection of $W^u\cap\{t=\nu\in \mathcal N\}$ with $\nu>0$ onto the $x,y$-plane. Then there exist an $\epsilon_0>0$ and a $\nu>0$, both small enough, such that for all $\epsilon\in (0,\epsilon_0)$:
\begin{align}
W^u(\nu,E)=\operatorname{span}\begin{pmatrix}
 X(\epsilon^{1/3},E) \\
  -{\sqrt{\mu (\nu,E)}}Y(\epsilon^{1/3},E)
\end{pmatrix},\eqlab{WuExpr2}
\end{align}
where 
\begin{align*}
 X(\epsilon^{1/3},E) &= \cos\left(\frac{1}{\epsilon }\int_{0}^{\nu} \sqrt{\mu(s,E)}ds -\frac{\pi}{4}+ \epsilon^{2/3} \phi_1(\epsilon^{1/3},E)\right),\\
 Y(\epsilon^{1/3},E) &= 
 (1+\epsilon^{2/3}\rho(\epsilon^{1/3},E))\sin\left(\frac{1}{\epsilon }\int_{0}^{\nu} \sqrt{\mu(s,E)}ds -\frac{\pi}{4}+ \epsilon^{2/3} \phi_2(\epsilon^{1/3},E)\right)
\end{align*}
with 
$\rho,\phi_1,\phi_2:[0,\epsilon_0^{1/3})\times D\rightarrow \mathbb R$ all $C^M$-smooth.
\end{theorem}}
%
%
%
%

\begin{remark}
\response{
\eqref{WuExpr} follows from \eqref{WuExpr2} in the smooth setting with $\mu(t,E)=\mu(t)$, but \eqref{WuExpr} is the more familiar form, which we state with the least degree of smoothness.}
\response{Now, regarding finite smoothness, there is also a $C^k$-version of \thmref{mainTP2} (for $k$ large enough). In particular, for any $M\in \mathbb N$ there is a $k\in \mathbb N$ such that $\mu\in C^k$ suffices for the statement. However, determining $k$ as a function of $M$ requires some additional bookkeeping (see \remref{bookkeeping} in the appendix), and we therefore leave out such a statement. 

It is also possible that $\epsilon_0>0$ and $\nu>0$ can be chosen independent of $M$ (so that $\rho,\phi_1,\phi_2$ are $C^\infty$), but we have not pursued this. In fact, \eqref{WuExpr2} is valid for any $\nu>0$ for which $\mu(t,E)>0$ for all $t\in (0,\nu]$. This can be seen from the proof, but can also be obtained by extending \eqref{WuExpr2} using (a smooth version of) \eqref{xytxyt03}. With this extension, it also follows that the $\rho,\phi_1,\phi_2$ are $C^M$-smooth jointly in $\epsilon^{1/3},E$ and $\nu\ge c>0$.
}
\end{remark}
\begin{remark}
 \response{In the case when $\mu(t,E)\equiv t$, then \eqref{system1} reduces to the Airy equation \eqref{xytairy0} and the solution \eqref{airytrack} provides the desired tracking of $W^u$ across the turning point. In particular, upon using  \eqref{Aitau_} and  \eqref{AiPtau_}, we obtain the following form of $W^u(\nu)$
  \begin{align}\eqlab{Wuairy}
  W^u(\nu)=\operatorname{span}\begin{pmatrix}
 \cos\left(\frac23 \epsilon^{-1} \nu^{3/2}-\frac{\pi}{4}+\mathcal O(\epsilon)\right) \\
  -\nu^{1/2} \sin\left(\frac23 \epsilon^{-1} \nu^{3/2}-\frac{\pi}{4}+\mathcal O(\epsilon)\right)\left(1+\mathcal O(\epsilon)\right)
\end{pmatrix}.
 \end{align}
Although each of the remainder terms $\mathcal O(\epsilon)$-terms is smooth (Gevrey-1) functions of $\epsilon$ for $\mu(t,E)\equiv t$, \eqref{Wuairy} is still in agreement with \thmref{mainTP2}.  In general, the remainder terms will only be smooth functions $\epsilon^{1/3}$. (In fact, we can show (with some extra effort) that the functions $\rho,\phi_1,\phi_2$ are smooth functions of $\epsilon^{2/3}$, $\epsilon$ and $E$. However, this requires some additional bookkeeping, and for this reason we have settled with the simpler version of the theorem. ) }
\end{remark}

 To prove these results, we consider the extended system:
\begin{equation}\eqlab{system1Ext}
\begin{aligned}
 \dot x &=y,\\
 \dot y &=-\mu(t) x,\\
 \dot t&=\epsilon,\\
 \dot \epsilon &=0.
\end{aligned}
\end{equation}
Here the set of points $(x,0,0,0)$ is completely degenerate, the linearization having only zero eigenvalues. We therefore apply the following blowup transformation:
\begin{align}
 \Phi:\quad (r,(\bar y,\bar t,\bar \epsilon))\mapsto \begin{cases}
                                           y&=r\bar y,\\
                                           t&=r^2 \bar t,\\
                                           \epsilon &=r^3 \bar \epsilon,
                                          \end{cases}\eqlab{blowup}
\end{align}
where $r\ge 0$, $(\bar y,\bar t,\bar \epsilon)\in S^2$, $S^2\subset \mathbb R^3$ being the unit sphere, with the purposes of gaining improved properties of the linearization. The preimage of any point $(x_0,0,0,0)$ under \eqref{blowup} becomes a sphere defined by $x=x_0$, $r=0$. Geometrically, the inverse process of \eqref{blowup} therefore blows up the degenerate line defined by $(x,0,0,0)$ to a cylinder $\mathbb R\times S^2$ defined by $x\in \mathbb R$, $r=0$. We illustrate the blowup in \figref{blowup}.

  \begin{figure}[!ht] 
\begin{center}
{\includegraphics[width=.85\textwidth]{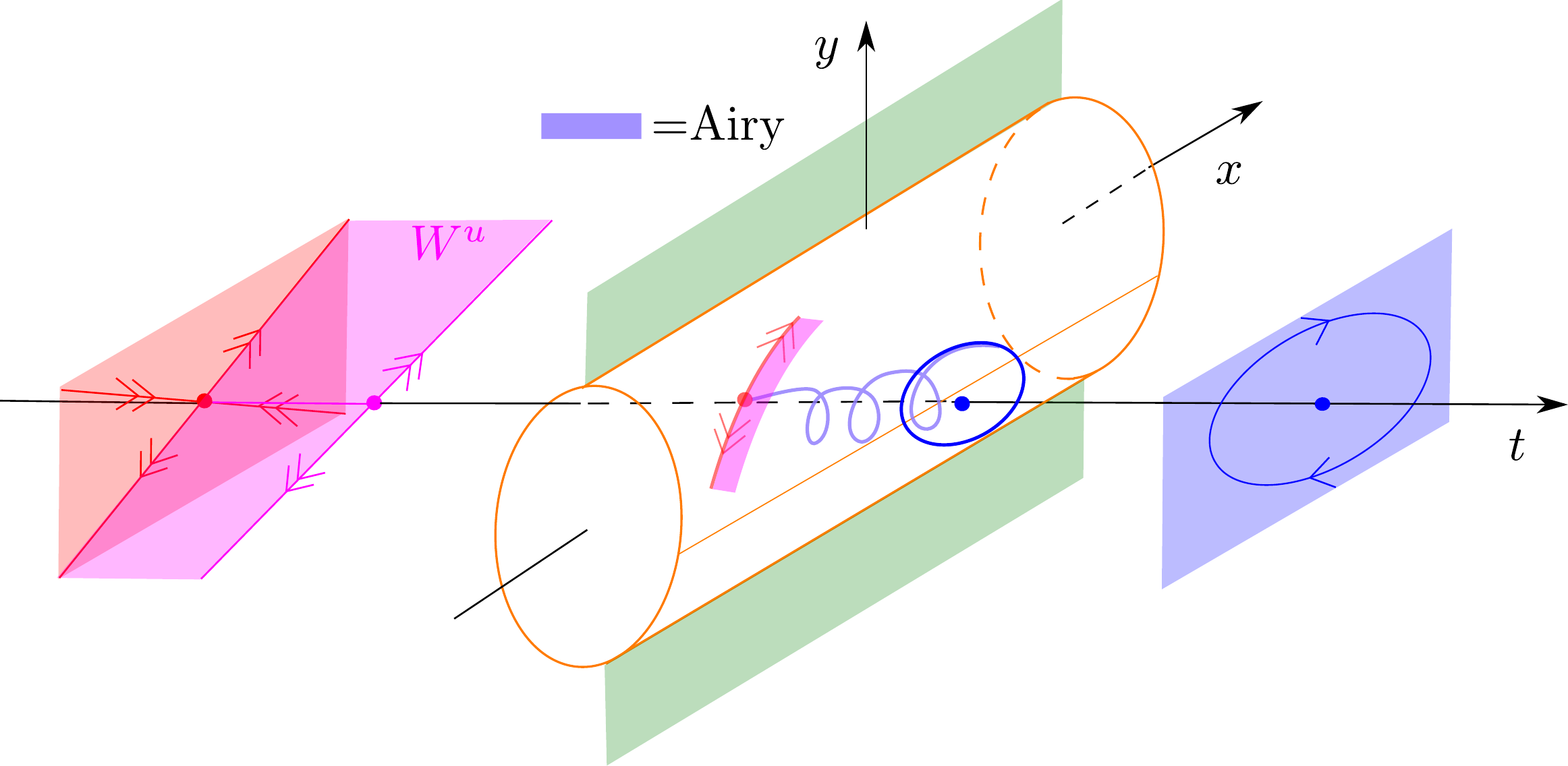}}
\end{center}
 \caption{The blowup transformation defined by \eqref{blowup} replaces the degenerate $x$-axis (orange in \figref{xyt}) with a cylinder of spheres. The blowup allow us to extend the hyperbolic and oscillatory regimes on either side. The connection is obtain by the Airy-function (corresponding orbit shown in purple). We refer to the text for further details.   }
\figlab{blowup}
\end{figure}

Let $X$ denote the vector-field defined by \eqref{system1Ext}. Then the exponents (or weights) on $r$ in \eqref{blowup} are chosen so that $\overline X=\Phi^* X$ has a power of $r$ as a common factor. In this case we have
\begin{align*}
 \overline X = r \widehat X,
\end{align*}
and it is $\widehat X$, the desingularized vector-field, that we shall study in the following. 

To study $\widehat X$ we follow \cite{krupa_extending_2001} and use directional charts $\{\bar t=-1\}$, $\{\bar \epsilon=1\}$ and $\{\bar t=1\}$, introducing chart-specified coordinates in the following way:
 \begin{align}
\{\bar t=-1\}:\quad &\begin{cases}
y&=r_1 y_1,\\
t &= -r_1^2,\\
\epsilon &=r_1^3 \epsilon_1,
\end{cases}\eqlab{chartT_1TP}\\
\{\bar \epsilon=1\}:\quad &
\begin{cases}
y&=r_2 y_2,\\
t &= r_2^2t_2,\\
\epsilon &=r_2^3,
\end{cases}\eqlab{chartE1TP}\\
\{\bar t=1\}:\quad &
\begin{cases}
y&=r_3 y_3,\\
t &= r_3^2,\\
\epsilon &=r_3^3 \epsilon_3.
\end{cases}\eqlab{chartT1TP}
\end{align}
The coordinate changes between ``adjacent'' charts are given by the following equations:
\begin{align}
 &\begin{cases} r_1 &= r_2 (-t_2)^{1/2},\\  y_1 &= y_2 (-t_2)^{-1/2},\\\epsilon_1 &= (-t_2)^{-3/2},\end{cases}\eqlab{cc12TP}\\
 &\begin{cases}r_3 &= r_2 t_2^{1/2},\\  y_3 &= y_2 t_2^{-1/2},\\ \epsilon_3 &= t_2^{-3/2}.\end{cases}\eqlab{cc23TP}
\end{align}
%
Before presenting the full details, we summarize our approach and our findings:
In the chart $\{\bar t=-1\}$, we gain hyperbolicity along $r_1=\epsilon_1=0$ and can -- by standard hyperbolic theory and upon applying the coordinate transformation -- track the unstable manifold $W^u$ into the scaling chart $\{\bar \epsilon=1\}$. Notice that \eqref{chartE1TP} corresponds to \eqref{scaling2} upon eliminating $r_2=\epsilon^{1/3}$. In the $\{\bar \epsilon=1\}$-chart, we therefore obtain a regular perturbation problem of the Airy-equation. The solution $x(t_2)=\operatorname{Ai}\,(-t_2)$ of \eqref{airyeqn} therefore provides a tracking of $W^u$ up to $t_2=\delta^{-2/3}$, with $\delta>0$ small, say, for any $0<r_2\ll 1$. Upon applying the coordinate transformation \eqref{cc23TP}, we therefore have an accurate description of the unstable manifold $W^u$ in the chart $\{\bar t=1\}$ for $\epsilon_3\ge \delta$. However, in this chart we only gain ellipticity along $r_3=\epsilon_3=0$, in particular $\dot x=y_3,\,\dot y_3=-x$ within $r_3=\epsilon_3=0$, and this is where the main difficulty of the proof lies. 
Essentially, our approach is to push the normalization procedure of \lemmaref{diagsimple}, see also \lemmaref{uvN}, into the the chart $\{\bar t=1\}$ and in this way accurately track the unstable manifold from the scaling chart into $t>0$ as desired. 
The fact that our $\mu$ is not necessarily analytic is not problematic. We will only use that the Taylor polynomials of $\mu$ are analytic to obtain a ``quasi-diagonalization'' that diagonalizes up to remainder terms of the order $r_3^{2L} \epsilon_3^M$ with $L,\,M\in \mathbb N$. Here we will rely on a ``center manifold''--version of the result in \cite{canalis-durand2000a,de2020a} used for the existence of $f(t,\epsilon)$ above, see details in \lemmaref{center} below. The reason for using quotation marks around center manifold is that formally we are dealing with a situation where the eigenvalues are $i\omega, 0$ and hence the center space is in fact the full space. Instead, our ``center manifold'', which we shall call an elliptic center manifold in the following, will be an invariant manifold that is a graph over the zero eigenspace, in much the same way as the slow manifold defined by $f(t,\epsilon)$ above. The fact that we rely on a center-like  version of the singular perturbation result used in \lemmaref{diagsimple} is not surprising.  Indeed center manifold theory is already central in the extension of slow manifolds in the classical hyperbolic setting \cite{dumortier1996a,krupa_extending_2001}.

\section{Proofs of main results}\seclab{proofTP}
To prove \thmref{mainTP} and \thmref{mainTP2}, we consider \eqref{system1Ext} and apply the blowup transformation \eqref{blowup}. We study the desingularized vector-field $\widehat X$ by working in the directional charts $\{\bar t=-1\}$, $\{\bar \epsilon=1\}$ and $\{\bar t=1\}$. As different aspects of the proof require different smoothness properties, we will assume that $\mu \in C^k$ and state necessary smoothness conditions in each partial result. Moreover, the solution in \thmref{mainTP} item \ref{expansionsol} will be parametrized in different ways in the separate charts. It should be clear from the context how solutions are related by re-parametrizations.
\subsection{Analysis in chart $\{\bar t=-1\}$}
  Inserting \eqref{chartT_1TP} into the extended system \eqref{system1Ext} gives the following equations:
  \begin{equation}\eqlab{eqns1}
  \begin{aligned}
   \dot x &=y_1,\\
   \dot y_1 &=\widehat \mu(-r_1^2)x+\frac12 \epsilon_1 y_1,\\
   \dot r_1 &=-\frac12 r_1 \epsilon_1,\\
    \dot \epsilon_1 &=\frac{3}{2}\epsilon_1^2,
  \end{aligned}
  \end{equation}
  upon division of the right hand side by the common factor $r_1$, recall \eqref{muhat}. If $\mu \in C^k$ then $\widehat \mu\in C^{k-1}$.

Now, setting $r_1=\epsilon_1=0$ gives 
\begin{align*}
 \dot x &=y_1,\\
 \dot y_1 &= x.
\end{align*}
For this system $(x,y_1)=(0,0)$ is a linear saddle with stable space $\operatorname{span}(1,-1)$ and unstable space $\operatorname{span}(1,1)$. For the full $(x,y_1,r_1,\epsilon_1)$-space, the points $(0,0,r_1,\epsilon_1)$ with $r_1$, $\epsilon_1\ge 0$ sufficiently small, is a center manifold having a smooth foliation of stable and unstable fibers. Each foliation produces a stable manifold  $W_1^s$ and an unstable manifold $W_1^u$.
\begin{lemma}\lemmalab{Wu1}
%
Suppose $\mu\in C^{k}$ with $k\ge 3$. Then there exists a $\upsilon>0$ sufficiently small, so that the manifolds $W_1^{s}$ and $W_1^u$ (line bundles over $x=y_1=0$) take the following graph form
\begin{align*}
W_1^s:\quad &y_1 = h_s(r_1^2,\epsilon_1)x,\\
 W_1^u:\quad &y_1 = h_u(r_1^2,\epsilon_1)x,
\end{align*}
over $(r_1,\epsilon_1)\in [0,\upsilon]^2$. The functions $h_{s,u}$ are $C^{k-1}$-smooth and satisfy $h_s(0,0)=-1$, $h_u(0,0)=1$
Finally, $W_1^s\cap \{r_1=0\}$ is nonunique whereas $W_1^u\cap \{r_1=0\}$ is unique  and given by the span of $\operatorname{Ai}$:
\begin{align}
 W_1^u\cap \{r_1=0\}:\quad  \begin{cases} x(\epsilon_1) &= l\operatorname{Ai}(\epsilon_1^{-2/3}),\\
 y_1(\epsilon_1)& = l\epsilon_1^{1/3}\operatorname{Ai}'(\epsilon_1^{-2/3})
 \end{cases},l\in \mathbb R,\epsilon_1\ge 0.\eqlab{Wu1r0}
\end{align}
%
%
\end{lemma}
\begin{proof}
The existence and smoothness of $W_1^{s/u}$ follow from center manifold theory \cite{car1}. In fact, for the expansion of $h_{s,u}$ we follow the proof of \lemmaref{WutNeg} and use projective coordinates $u:=x^{-1} y_1$ with
\begin{align*}
 \dot u &=\widehat \mu(-r_1^2)+\frac12 \epsilon_1 u-u^2.
\end{align*}
Here $W_1^s$ and $W_1^u$ correspond to center manifolds $u=h_{s,u}(r_1^2,\epsilon_1)$ of $(u,r_1,\epsilon_1)=(\mp 1,0,0)$, respectively. Since $\widehat \mu\in C^{k-1}$ these manifolds are also $C^{k-1}$ for any $k\ge 3$, see \cite{car1}. Regarding the (non-)uniqueness of $W_1^{s,u}$ within $r_1=0$, we have
\begin{align*}
 \dot u &=1+\frac12 \epsilon_1 u-u^2,\\
 \dot \epsilon_1 &=\frac32 \epsilon_1^2.
\end{align*}
Since $\epsilon_1$ is increasing for $\epsilon_1>0$, we have that $(u,\epsilon_1)=(1,0)$ is a nonhyperbolic saddle with the center manifold coinciding with the unstable set. $W_1^u$ is therefore unique in this case. $(u,\epsilon_1)=(-1,0)$ on the other hand is a nonhyperbolic unstable node and consequently, $W_1^s$ is nonunique. 


From the asymptotics of $\text{Ai}$, see \eqref{Aitau}, it is clear that the right hand side of \eqref{Wu1r0} belongs to $W_1^u\cap \{r_1=0\}$ and since it is unique the result follows. This completes the proof.

\end{proof}


Upon blowing down, $W_1^u$ provides a desired extension of the unstable manifold $W^u$.
On  $W_1^u$ we have $\dot x = h_u(r_1^2,\epsilon_1)x$ and therefore upon dividing by $\dot r_1$ we obtain
\begin{align}
 \frac{dx}{dr_1} &=\frac{-2}{r_1 \epsilon_1(r_1)}h_s(r_1^2,\epsilon_1(r_1)) x,\eqlab{expx1}
\end{align}
with $\epsilon_1(r_1)=r_1^{-3} \epsilon$ for some $0<\delta<\upsilon$. Here we have used the conservation of $\epsilon=r_1^3 \epsilon_1$. We now describe solutions on $W_1^u$. For this purpose it is convenient to expand $h_u$ in the following form
\begin{align}
 h_u(r_1^2,\epsilon_1)= U_1(\epsilon_1)+U_2(r_1^2)+\epsilon_1 U_3(r_1^2)+r_1^2 \epsilon_1^2 U_4(r_1^2,\epsilon),\eqlab{hu}
\end{align}
where
\begin{align*}
 U_1(\epsilon_1)&:=h_u(0,\epsilon_1) =1+\frac14 \epsilon_1 +\mathcal O(\epsilon_1^2),\\
 U_2(r_1^2)&:=h_u(r_1^2,0)-h_u(0,0)=\sqrt{\widehat \mu(-r_1^2)}-1,\\
 U_3(r_1^2)&:=\frac{\partial}{\partial \epsilon_1} h_u(r_1^2,0)-\frac{\partial}{\partial \epsilon_1} h_u(0,0)=-\frac14 r_1^2 \widehat \mu(-r_1^2)^{-1} \widehat \mu'(-r_1^2).
\end{align*}
With $U_1,U_2,U_3$ defined, this expansion just follows from a Taylor-expansion of the difference $$h_u(r_1^2,\epsilon_1)- \left(U_1(\epsilon_1)+U_2(r_1^2)+\epsilon_1 U_3(r_1^2)\right).$$ Notice that $U_2(r_1^2),U_3(r_1^2)=\mathcal O(r_1^2)$ but also that $U_1\in C^\infty$ (even Gevrey-1 analytic) since the $r_1=0$-subsystem is analytic, whereas $U_i$, $i=2,3$ are $C^{k-1}$ smooth. $U_4$ is $C^{k-2}$-smooth. Since $y_1=U_1(\epsilon_1)x$ gives $W_1^u\cap \{r_1=0\}$ we also obtain the following convinient expressions for $U_1$ in terms of the \text{Ai}-function using \eqref{Wu1r0}:
\begin{cor}\corlab{U1Ai}
 $U_1(\epsilon_1)=\operatorname{Ai}(\epsilon_1^{-2/3})^{-1} \epsilon_1^{1/3}\operatorname{Ai}'(\epsilon_1^{-2/3})=\frac32 \epsilon_1^2 \left(\log\operatorname{Ai}(\epsilon_1^{-2/3})\right)'.$ 
\end{cor}
%
We then obtain the following, which in turn proves  \thmref{mainTP} item (\ref{expansion1}) upon using \eqref{chartT_1TP}. 
\begin{lemma}\lemmalab{solchart1}
For every $0<\epsilon\ll 1$, there exists a solution $$(x(r_1),y_1(r_1),r_1,\epsilon_1(r_1))\in W_1^u,$$ of \eqref{eqns1} (using $r_1$ as time) of the following form:
 \begin{align*}
  x(r_1) =& {\widehat \mu(r_1^2)}^{-1/4}
\operatorname{Ai}(\epsilon_1(r_1)^{-2/3}) e^{-\frac{1}{\epsilon} \int_{0}^{r_1^2}\left(\sqrt{-\mu(-s^2)}-s\right) ds^2}\left(1+\mathcal O(\epsilon^{2/3})\right),\\
  y_1(r_1) =&\epsilon_1(r_1)^{1/3}{\widehat \mu(r_1^2)}^{1/4} \operatorname{Ai}'(\epsilon_1(r_1)^{-2/3})
e^{-\frac{1}{\epsilon} \int_{0}^{r_1^2}\left(\sqrt{-\mu(-s^2)}-s\right) ds^2}\times \\
&\left(1+\mathcal O(\epsilon^{2/3})\right)\\
  \epsilon_1(r_1) &=r_1^{-3}\epsilon,
  \end{align*}
  for $r_1\in [\epsilon^{1/3}\delta^{-1/3},\nu]$, 
\end{lemma}
\begin{proof}
We simply integrate \eqref{expx1} from $r_{1,\textnormal{out}}:=\epsilon^{1/3} \delta^{-1/3}$ to $r_1$. This gives
\begin{align*}
 x(r_1) = &\exp\left(\frac23 \int \epsilon_1^{-2} U_1(\epsilon_1) d\epsilon_1 \right)\exp\left(-\frac{2}{\epsilon} \int_{r_{1,\textnormal{out}}}^{r_1} s^{2} U_2(s^2) ds \right)\times \\
 &\exp\left(\frac12 \int r_1 \widehat \mu(r_1^2) \widehat \mu'(r_1^2) dr_1 \right)\exp\left( -2\epsilon \int_{r_{1,\textnormal{out}}}^{r_1} s^{-2} U_4(s^2,s^{-3} \epsilon)ds\right)\\
 =&\operatorname{Ai}(\epsilon_1(r_1)^{-2/3})\exp\left(-\frac{1}{\epsilon} \int_{r_{1,\textnormal{out}}^2}^{r_1^2} \left(\sqrt{-\mu(-s^2)}-s\right) ds^2 \right)\times \\
 &\widehat \mu(-r_1^2)^{-1/4}\exp\left( -2\epsilon \int_{r_{1,\textnormal{out}}}^{r_1} s^{-2} U_4(s^2,s^{-3} \epsilon)ds\right).
\end{align*}
(In principle, the indefinite integrals should be definite ones also but the difference just produce constants that are independent of $\epsilon$ and which due to the linearity can be scaled out.)
Next, we use that $\vert U_4(r_1^2,r_1^{-3} \epsilon)\vert \le c$ for all $r_1\in [r_{1,\textnormal{out}},\nu]$ with $c>0$ large enough and
\begin{align}
\epsilon^{-1} \int_0^{r_{1,\textnormal{out}}^2} \mathcal O(s^{3}) ds^2 = \mathcal O(\epsilon^{2/3}),\quad  
 \epsilon \int_{r_{1,\textnormal{out}}}^{r_1} s^{-2} ds =\mathcal O(\epsilon^{2/3}),\eqlab{estsol1}
\end{align}
for $r_1\in [r_{1,\textnormal{out}},\nu]$.  
%
%
To obtain the expression for $y_1$, we first multiply by $h_u=U_1(1+U_1^{-1} U_2 +\mathcal O(\epsilon^{2/3})$ using $\vert \epsilon_1 U_3(r_1^2)\vert, r_1^2\epsilon_1^2 U_4(r_1^2,\epsilon_1)\vert =\mathcal O(\epsilon^{2/3})$. The  result then follows from \corref{U1Ai} and the fact that $1+U_1^{-1} U_2 = \sqrt{\widehat\mu(-r_1^2)}+\mathcal O(r_1^2\epsilon_1) = \sqrt{\widehat\mu(-r_1^2)}(1+\mathcal O(\epsilon^{2/3}))$. 
\end{proof}

\subsection{Analysis in chart $\{\bar \epsilon=1\}$}
Inserting \eqref{chartE1TP} into the extended system \eqref{system1Ext} gives the following equations:
\begin{equation}\eqlab{eqnk2}
\begin{aligned}
 \dot x &= y_2,\\
 \dot y_2 &=-t_2 \widehat \mu(r_2^2 t_2)x,\\
 \dot t_2 &=1,
\end{aligned}
\end{equation}
after division of the right hand side by the common factor $r_2=\epsilon^{1/3}$. Notice that in this chart, we have $r_2=\epsilon^{1/3}$ as a small parameter. In this way, following \eqref{cc12TP}, the unstable set $W_{1}^u$, from the chart $\{\bar t=-1\}$, becomes an $r_2$-family of two-dimensional manifolds, that we will denote by $W_{2}^u(r_2)$, within the $(x_2,y_2,t_2)$-space. Moreover, within compact subsets of the $(x_2,y_2,t_2)$-space, it is a smooth $\mathcal O(r_2^2)$-perturbation of the unperturbed space $W_2^u(0)$ given by the set of points $(lx(t_2),ly(t_2),t_2)$, $l\in \mathbb R$, with $x_2(t_2)=\operatorname{Ai}\,(-t_2)$, $y_2(t_2)=-\operatorname{Ai}'\,(-t_2)$ as 
a solution of the unperturbed system:
\begin{equation}\eqlab{airyunscaled}
\begin{aligned}
 \frac{dx_2}{dt_2} &=y_2,\\
 \frac{dy_2}{dt_2}&=-t_2 x_2.
\end{aligned}
\end{equation}


Consider the solution described in \lemmaref{solchart1} in the chart $\{\bar t=-1\}$. Upon using \eqref{cc12TP}, we can extend the solution in the coordinates $(x(t_2),y(t_2),t_2)$ of the $\{\bar \epsilon=1\}$-chart for any $t_2\in [-\delta^{-1},\delta^{-1}]$. We have
\begin{lemma}\lemmalab{solchart2}
The following holds:
\begin{align*}
 x(t_2) &=\operatorname{Ai}(-t_2)+\mathcal O(r_2^2),\\ 
 y_2(t_2)&=-\operatorname{Ai}'(-t_2)+\mathcal O(r_2^2),
\end{align*}
uniformly within $t_2\in [-\delta^{-1},\delta^{-1}]$. 
\end{lemma}
\begin{proof}
 Follows from regular perturbation theory and a straightforward calculation based upon \eqref{cc12TP} and \lemmaref{solchart1}. 
\end{proof}
This result proves \thmref{mainTP} item (\ref{expansion2}). 
Notice that $W_2^u(r_2)$ is given by the set of points $(lx(t_2),ly(t_2),t_2)$, $l\in \mathbb R$, with $x(t_2),y(t_2)$ described by \lemmaref{solchart2}.

\subsection{Analysis in the chart $\{\bar t=1\}$}\seclab{texit}
Inserting \eqref{chartT1TP} into the extended system \eqref{system1Ext} gives the following equations:
\begin{equation}\eqlab{eqnk3}
  \begin{aligned}
   \dot x &=y_3,\\
   \dot y_3 &=-\widehat \mu(r_3^2)x-\frac12 \epsilon_3 y_3,\\
   \dot r_3 &=\frac12 r_3 \epsilon_3,\\
    \dot \epsilon_3 &=-\frac{3}{2}\epsilon_3^2,
  \end{aligned}
  \end{equation}
  upon division of the right hand side by the common factor $r_3$.  Upon using the coordinate change \eqref{cc23TP}, we realize that the solution in \lemmaref{solchart2} takes the following form in $\{\bar t=1\}$:
  \begin{equation}
  \begin{aligned}
   x(\epsilon_3) &= \operatorname{Ai}\,(-\epsilon_3^{-2/3})+\mathcal O(r_3^2),\\
y_3(\epsilon_3)&=\epsilon_3^{1/3}\left(-\operatorname{Ai}'\,(-\epsilon_3^{-2/3})+\mathcal O(r_3^2)\right),
\end{aligned}\eqlab{Wr3}
\end{equation}
and $r_3 = \epsilon^{1/3} \epsilon_3^{-1/3}$ for $\epsilon_3\ge \delta$. Let $W_3^u$ denote the unstable manifold in the present chart. It is spanned by $(lx(\epsilon_3),ly_3(\epsilon_3),r_3,\epsilon_3)$, $l\in \mathbb R$, within $\epsilon_3\ge \delta$. Due to the $C^{k-1}$-smoothness of $W_1^u$ in the chart $\{\bar t=-1\}$, the dependency on $r_3^2$ in the expressions of \eqref{Wr3} are also $C^{k-1}$-smooth. We now extend the solution \eqref{Wr3} and $W_3^u$ up to $r_3=\nu$; this will allow us to finish the proof of \thmref{mainTP}.

  Let 
    $$\lambda_3(r_3^2):= i \sqrt{\widehat \mu(r_3^2)},$$ so that $\overline \lambda_3=-\lambda_3$.
  Then in the $(x,y_3,r_3,\epsilon_3)$-space, the linearization about any equilibrium point $(0,0,r_3,0)$, with $r_3\ge 0$ sufficiently small, has $\pm \lambda_3(r_3^2)$ as the only nonzero eigenvalues. 
    In particular, consider the invariant $\epsilon_3=0$ subspace:
  \begin{equation}\eqlab{eqnk3_eps30}
   \begin{aligned}
     \dot x &= y_3,\\
   \dot y_3 &=-\widehat \mu(r_3^2) x,\\
   \dot r_3 &=0.
  \end{aligned}
  \end{equation}
Obviously, if we define $u$ and $v$ by
\begin{align}\eqlab{uveps30}
 \begin{pmatrix}
  x\\
  y_3
 \end{pmatrix}=
 \begin{pmatrix}
 1 & 1\\
 \lambda_3(r_3^2) & -\lambda_3(r_3^2)
\end{pmatrix}\begin{pmatrix}
             u\\
             v
            \end{pmatrix},
\end{align}
then \eqref{eqnk3_eps30} becomes diagonalized
\begin{align*}
 \dot u &=\lambda_3(r_3^2)u,\\
 \dot v &=-\lambda_3(r_3^2) v,\\
 \dot r_3 &=0.
\end{align*}
As advertised below \thmref{mainTP2}, we now extend, in line with the proof of \lemmaref{diagsimple}, this diagonalization procedure (in an ``approximative'' way) into $\epsilon_3>0$.
 \begin{lemma}\lemmalab{k3transformation}
  Fix any $N\in \mathbb N$ such that $k\ge 2N+3$ where $k$ is the degree of smoothness of $\mu$. Then there exist a locally defined function $f_N(r_3^2,\epsilon_3)$ being $C^{k-N-1}$ and Gevrey-1 with respect to $\epsilon_3$ uniformly in $r_3\ge 0$, that satisfies $f_N(r_3^2,0)=1$ for all $r_3$, such that  the transformation of the following form
 \begin{align}
 \begin{pmatrix}
  x\\
  y_3
 \end{pmatrix} = \begin{pmatrix}
 f_N(r_3^2,\epsilon_3) &\overline{f}_N(r_3^2,\epsilon_3)\\
 \lambda_3(r_3^2) & \overline{\lambda}_3(r_3^2)
 \end{pmatrix}\begin{pmatrix}
  u\\
  v
 \end{pmatrix},\eqlab{uv3}
 \end{align}
brings \eqref{eqnk3} into the near-diagonalized form
\begin{equation}\eqlab{uvk3}
\begin{aligned}
 \dot u &=\nu_{N}(r_3^2,\epsilon_3) u+ \mathcal O(r_3^{2(N+1)}\epsilon_3^{N+1})v,\\
 \dot v &=\mathcal O(r_3^{2(N+1)}\epsilon_3^{N+1})u+\overline{\nu}_{N}(r_3^2,\epsilon_3) v.
\end{aligned}
\end{equation}
Here 
\begin{align*}
 \nu_{N} &=\lambda_3(r_3^2) f_N(r_3^2,\epsilon_3) - \frac12 \epsilon_3 \left(1+2r_3^2 \lambda_3(r_3^2)^{-1}\lambda_3'(r_3^2)\right)+\mathcal O(r_3^{2(N+1)}\epsilon_3^{N+1}).  
\end{align*}
\response{as well as the remainder terms $\mathcal O(r_3^{2(N+1)}\epsilon_3^{N+1})$ in \eqref{uvk3},} are both $C^{k-N-2}$ and Gevrey-1 with respect to $\epsilon_3$ uniformly in $r_3^2\ge 0$.

In particular, we have the following regarding $f_N$:

Let $B_0(\epsilon_1):=f_N(0,\epsilon_1)$. Then $B_0$ satisfies
\begin{align*}
\frac32 \epsilon_3^2 B_0'&=i(B_0^2-1)-\frac12 \epsilon_3 B_0
\end{align*}
with $B_0(0)=1$, $B_0'(0)=-\frac{i}{4}$, and is Gevrey-$1$, whereas 
\begin{align*}
R_1(r_3^2):=\frac{\partial }{\partial \epsilon_3} f_N(r_3^2,0)=\frac14 \lambda_3(r_3^2)^{-1} \left(1+ 2r_3^2 \lambda_3(r_3^2)^{-1}\lambda_3'(r_3^2)\right),
\end{align*}
with $R_1(0) = B_0'(0)=-\frac{i}{4}$, is $C^{k-2}$ smooth in $r_3^2$.
\end{lemma}

\begin{proof}
%
We will address the required finite smoothness towards the end of the proof. Until then we will assume that $k=\infty$.  
In line with \eqref{uveps30}, we seek the desired transformation in the form \eqref{uv3}
with $f_N(r_3^2,0)=1$. For simplicity we now drop the subscripts $3$ and $N$.

Inserting \eqref{uv3} into \eqref{eqnk3} shows that the $u,v$-system is diagonalized if $f$ with $f(r^2,0)=1$ satisfies
the singular PDE:
\begin{align}
 \frac32 \epsilon^2 f'_{\epsilon} -r^2 \epsilon f'_{r^2} = \lambda \left( f^2 -1\right)-\frac12 \epsilon\left(1+2r^2 \lambda^{-1} \lambda'\right) f.
\eqlab{f3eqn}
\end{align}
\response{By the method of characteristics, this system can be written as a first order system
\begin{align*}
 \dot f &= \lambda(1-f^2) +\frac12 \epsilon (1+2r^2 \lambda^{-1}\lambda')f,\\
 \dot r &=\frac12 r \epsilon,\\
\dot \epsilon&=-\frac32 \epsilon^2.
\end{align*}
This system is clearly also the extended system: \eqref{ft2}, $\dot \epsilon=0$ written in the $\bar t=1$-chart.}

Here and in the following, we use the notation $g'_x = \frac{\partial g}{\partial x}$ for partial derivatives of a smooth function $g$ with respect to $x$. 
We will not aim to solve \eqref{f3eqn} exactly (which is even an open problem for \eqref{ft} in the smooth setting anyway) and will instead obtain appropriate ``quasi-solutions'' of \eqref{f3eqn}. We define these functions as follows: 
Write \eqref{f3eqn} as
\begin{align*}
 F(f,r^2,\epsilon) = 0.
\end{align*}
We then say that $f(r^2,\epsilon)$ (of the class described in the lemma) is a ``quasi-solution'' of order $\mathcal O(r^{2L} \epsilon^M)$ with $L,M\in \mathbb N$ if 
\begin{align*}
 F(f(r^2,\epsilon),r^2,\epsilon) = \mathcal O(r^{2L}\epsilon^M).
\end{align*}
Notice, due to singular nature of the left hand side of \eqref{f3eqn}, that if $f(r^2,\epsilon)$ is a quasi-solution of order $\mathcal O(r^{2L} \epsilon^M)$, then so is any smooth function of the form $f(r^2,\epsilon)+\mathcal O(r^{2L}\epsilon^M)$ (we gain factors of $r^2$ and $\epsilon$, that we lose upon differentiation, through the factors of the partial derivatives).
%
In the following we will construct two separate ($\textcolor{red}{R}$) and ($\textcolor{blue}{B}$) ``quasi-solutions'', in this sense, that solve \eqref{f3eqn} up to order $\mathcal O(\epsilon^M)$ and $\mathcal O(r^{2L})$, respectively. Then upon combining these transformations, in an appropriate (and obvious) way, we achieve the desired  ``quasi-solution'' of \eqref{f3eqn} leaving a remainder of order $\mathcal O(r^{2(N+1)} \epsilon^{N+1})$ (upon setting $L=M=N+1$). 

For the first quasi-solution ($\textcolor{red}{R}$), we insert 
\begin{align}
 f = \sum_{m=0}^{M-1} R_m(r^2)\epsilon^m,\eqlab{fquasiR}
\end{align}
with $R_0(r^2)=1$ for all $r^2$ into \eqref{f3eqn}. 
This gives
\begin{align}
  \frac{3}{2}\sum_{m=1}^{M-1} (m-1) R_{m-1} \epsilon^m &=r^2 \sum_{m=1}^{M-1} R_{m-1}'\epsilon^m + \lambda(r^2) \sum_{m=1}^{M-1} \left(2R_m+\sum_{l=1}^{m-1} R_l R_{m-l}\right) \epsilon^m \nonumber\\
 &-  \frac12 \left(1+2r^2\lambda^{-1} \lambda'\right) \sum_{m=1}^{M-1} R_{m-1} \epsilon^m  + \mathcal O(\epsilon^{M}).
\eqlab{aaa}
\end{align}
Consequently, by collecting coefficients of $\epsilon^m$, $m=1,\ldots,M-1$, we obtain the recursive formula for the unknowns $R_m$:
\begin{align}
 R_m &= 
 -\frac12 \sum_{l=1}^{m-1} R_lR_{m-l}+\frac34 \lambda^{-1} (m-1)R_{m-1}-\frac12 \lambda^{-1} r^2R_{m-1}'\nonumber\\
 &-\frac14 \lambda^{-1} \left(1+2r^2\lambda^{-1}\lambda'\right)R_{m-1},
\eqlab{Amrec}
\end{align}
for $m=1,\ldots, M-1$ starting from $R_0=1$. 
In this way, \eqref{fquasiR} is a ``quasi-solution'' of order $\mathcal O(\epsilon^M)$. Notice we have the following: 
\begin{lemma}\lemmalab{gexpand}
{Let $g(r^2,\epsilon)$ be any smooth function whose $(M-1)$th order Taylor-polynomial with respect to $\epsilon$, having $r^2$-dependent coefficients, coincide with \eqref{fquasiR}. Then the composition $F(g(r^2,\epsilon),r^2,\epsilon)$ is of order $\mathcal O(\epsilon^M)$ uniformly with respect to $r^2$. }
\end{lemma}
\begin{proof}
We simply Taylor-expand $F(g(r^2,\epsilon),r^2,\epsilon)$ with respect to $\epsilon$. Since $g$ coincide with \eqref{fquasiR}, where $R_m$, $m=0,\ldots,M-1$ are precisely chosen to eliminate all terms up to order $\epsilon^{M-1}$, the result follows.
\end{proof}
Next, for the second quasi-solution ($\textcolor{blue}{B}$) we insert 
\begin{align}
 f &=\sum_{l=0}^{L-1}B_l(\epsilon)r^{2l},\eqlab{fquasiB}
\end{align}
with $B_0(0)=1$ and $B_l(0)=0$ for $l=1,\ldots, L-1$, 
into \eqref{f3eqn}. We expand all smooth functions in $r^2$ in power series. In this way, we obtain the following equation for $B_0$ 
\begin{align}
 \frac32 \epsilon^2 B_0'&=i(B_0^2-1)-\frac12 \epsilon B_0,\eqlab{B0rec}
\end{align}
and that all equations for $B_l$, $l=1,\ldots, L-1$ are linear in $B_l$, and can be written in the following form
%
%
%
\begin{align}
 \frac32 \epsilon^2 B_l'= 2iF_l(\epsilon,B_0,B_{1,\ldots,l-1}) B_l + G_l(\epsilon,B_0,B_{1,\ldots,l-1}), \eqlab{Bl1rec}
\end{align}
for $l=1,\ldots,L-1$. Here $F_l$ and $G_l$ are  polynomials satisfying $F_l(0,1,\textbf 0)=1$ and $G_l(0,1,\textbf 0)=0$. 
The equations \eqref{B0rec}, \eqref{Bl1rec} satisfy the conditions for a generalized (elliptic) center manifold of \cite[Proposition 6.2]{dema}, see also \cite[Theorem 12.1]{wasow1985a}, \cite[Chapter 8]{balser1994a} and  \cite[Theorem 1]{braaksma1992a}. In particular, there is formal power series solution $(\widehat B_0,\widehat B_1,\ldots,\widehat B_{L-1})$ of \eqref{B0rec}, \eqref{Bl1rec} which is $1$-summable in any direction except for the positive imaginary axes; see references for further details. This result is also a corollary of \cite[Theorem 3]{bonckaert2008a}. 
For our purposes, the following statement suffices. 
\begin{lemma}\lemmalab{center}
Suppose that  $F:U\rightarrow \mathbb C^n$ is an analytic function defined in an open neighborhood $U$ of $(0,0)\in \mathbb C\times \mathbb C^n$ satisfying
 \begin{align}
  F(0, 0)= 0,\,D_{y} F (0,0)=\textnormal{diag}(i\omega_1,\ldots,i\omega_n),\eqlab{centerAs}
 \end{align}
 with each $\omega_i\ne 0$,
and consider 
 \begin{align}
  x^{2} \frac{dy}{dx}=F(x, y),\eqlab{centerEq}
 \end{align}
or equivalently
\begin{align*}
 \dot x &=x^{2},\\
 \dot{y} &=F(x,y).
\end{align*}
Then there exists an $x_0>0$ and an invariant manifold of the graph form $ y=  Y(x)$, $x\in [0,x_0]$,  with $Y$ Gevrey-1, specifically $Y(0)=0$. 
%
\end{lemma}
Since the statement is crucial to our approach, we include a relatively simple proof -- based upon \cite{bonckaert2008a} and a fixed-point argument in the ``Borel-plane'' -- in \appref{proof}.
%
%
%

Applying \lemmaref{center} to \eqref{B0rec}, \eqref{Bl1rec} therefore shows the existence of a Gevrey-1 invariant manifold which we shall just write as $(B_0(\epsilon),B_1(\epsilon),\ldots,B_{L-1}(\epsilon))$, $\epsilon\in [0,\epsilon_0]$, satisfying $(B_0(0),B_1(0),\ldots,B_{L-1}(0))=(1,\textbf{0})$. 
We again point out the following simple result.
\begin{lemma}\lemmalab{hexpand}
{Let $h(r^2,\epsilon)$ be any smooth function whose $(L-1)$th order Taylor-polynomial with respect to $r^2$, having $\epsilon$-dependent coefficients, coincide with \eqref{fquasiB}. Then the composition $F(h(r^2,\epsilon),r^2,\epsilon)$ is of order $\mathcal O(r^{2L})$ uniformly with respect to $\epsilon$. }
\end{lemma}
\begin{proof}
We Taylor-expand the smooth function $F(h(r^2,\epsilon),r^2,\epsilon)$ now with respect to $r^2$. Since $h$ coincide with \eqref{fquasiB}, where $B_l$, $l=0,\ldots,L-1$ are precisely chosen to eliminate all terms up to order $r^{2(L-1)}$, the result follows.
\end{proof}

We then obtain our desired ``quasi-solution'' of \eqref{f3eqn} by setting
\begin{align}
 f(r^2,\epsilon) = \sum_{l=0}^{L-1} B_l(\epsilon)r^{2l}+\sum_{m=0}^{M-1} \left(R_m(r^2)-\sum_{j=0}^{L-1} \frac{R_m^{(j)}(0)}{j!} r^{2j}\right) \epsilon^m .\eqlab{f3quasi}
\end{align}
%
%
%
%
%
%
We illustrate the role of the terms in \eqref{f3quasi} in \tabref{tbl1}: 
The purpose of the first $L$ terms, due to the quasi-solution (\textcolor{blue}{$B$}), is to solve \eqref{f3eqn} exactly up to terms of order $r^{2L}$. We then remove terms (already accounted for) of order $r^{2l}$, $l=0,\ldots,L-1$, from the second quasi-solution (\textcolor{red}{$R$}), see the final term in \eqref{f3quasi}. Consequently, \eqref{f3quasi} satisfies the conditions of \lemmaref{gexpand}:
\begin{align}
 F(f(r^2,\epsilon),r^2,\epsilon) = \mathcal O(\epsilon^M),\eqlab{rem1}
\end{align}
as well as those in \lemmaref{hexpand}:
\begin{align}
 F(f(r^2,\epsilon),r^2,\epsilon) = \mathcal O(r^{2L}).\eqlab{rem2}
\end{align}
Hence, by Taylor expanding the remainder terms of \eqref{rem1} with respect to $r^2$ and using \eqref{rem2}, we obtain
\begin{align}
 F(f(r^2,\epsilon),r^2,\epsilon) = \mathcal O(r^{2L}\epsilon^M),\eqlab{rem3}
\end{align}
as claimed.
Upon setting $L=N+1$ and $M=N+1$, the function \eqref{f3quasi} therefore gives the desired ``quasi-solution'' of \eqref{f3eqn} of order $\mathcal O(r^{2L} \epsilon^M)=\mathcal O(r^{2(N+1)} \epsilon^{N+1})$ (the $\times$'s in \tabref{tbl1}).

This completes the proof of \lemmaref{k3transformation} in the $C^\infty$-case. In the case of finite smoothness, we recall that $\widehat \mu$ is $C^{k-1}$. Hence the equation $F(f,r^2,\epsilon)$ is $C^{k-2}$-smooth in $r^2$ due to the $\lambda_3'$-term.
Consequently, a sufficient condition for \eqref{rem2} to hold with $L=N+1$ and $f$ polynomial in $r^2$ as in \eqref{fquasiB} is that $k\ge N+3$. At the same time, it is clear that $R_1\in C^{k-2}$ and by induction we find $R_m\in C^{k-1-m}$ for each $m \le N\le k-1$. \eqref{f3quasi} is therefore $C^{k-N-1}$ and upon composing this with $F(\cdot,r^2,\epsilon)$, we obtain a
$C^{k-N-2}$-smooth function, losing one degree of smoothness due to the $f'_{r^2}$-term in the definiton of $F$. Consequently, a sufficient condition for the expansion in \eqref{rem3} with $L=M=N+1$ is that $k-N-2\ge N+1$, i.e. $k\ge 2N+3$. This completes the proof.

  \begin{table}[h]
    \renewcommand\arraystretch{2}
\begin{tabular}{|c|c|c|c|c|c|c|}
    \hline    
        $r_3^{2n}\backslash\epsilon_3^n$ & $\epsilon_3^0$ & $\epsilon_3$ & $\ldots$ &  $\epsilon_3^{M-1}$ & $\epsilon_3^{M}$ & $\ldots$\\
    \hline
$r_3^0$ & \textcolor{blue}{$\surd$} & \textcolor{blue}{$\surd$} & \textcolor{blue}{$\surd$} & \textcolor{blue}{$\surd$} & \textcolor{blue}{$\surd$} & \textcolor{blue}{$\surd$} \\
\hline 
$r_3^2$ & \textcolor{blue}{$\surd$} & \textcolor{blue}{$\surd$} & \textcolor{blue}{$\surd$} & \textcolor{blue}{$\surd$} & \textcolor{blue}{$\surd$} & \textcolor{blue}{$\surd$} \\
\hline
$\vdots $ & \textcolor{blue}{$\surd$} & \textcolor{blue}{$\surd$} & \textcolor{blue}{$\surd$} & \textcolor{blue}{$\surd$} & \textcolor{blue}{$\surd$} & \textcolor{blue}{$\surd$} \\
\hline
$r_3^{2(L-1)}$ & \textcolor{blue}{$\surd$} &\textcolor{blue}{$\surd$} & \textcolor{blue}{$\surd$} &\textcolor{blue}{$\surd$} & \textcolor{blue}{$\surd$} & \textcolor{blue}{$\surd$} \\
\hline
$r_3^{L}$ & \textcolor{red}{$\surd$} & \textcolor{red}{$\surd$} & \textcolor{red}{$\surd$} &\textcolor{red}{$\surd$} & $\times$ & $\times$ \\
\hline
$\vdots$ & \textcolor{red}{$\surd$} &\textcolor{red}{$\surd$} & \textcolor{red}{$\surd$} &\textcolor{red}{$\surd$} & $\times$ & $\times$ \\
\hline
\end{tabular}
\caption{Illustration of our strategy for diagonalization in the entry chart $\{\bar t=1\}$. See text for further description.}
\tablab{tbl1}
    \end{table}

\end{proof}
 
 \begin{remark}
  The method we use is akin to normal form theory. The linearization of the system \eqref{eqnk3} has eigenvalues $\pm i$, $0,0$ (semi-simple) and by normal form theory, see e.g. \cite{haragus2011a}, it follows that for each $M\in \mathbb N$ there is a linear mapping $(x,y_3)\mapsto (q,p)$, with coefficients that are each $N$th degree polynomials in $(r_3^2,\epsilon_3)$, such that 
  \begin{align*}
   \dot q  &= \widetilde \nu_N(r_3^2,\epsilon_3) p,\\
   \dot p &= -\widetilde\nu_N (r_3^2,\epsilon_3) q,
  \end{align*}
  with $\widetilde \nu_N$ polynomial of degree $N$, 
  up to remainder terms of order $\mathcal O(\vert (r_3^2,\epsilon_3)\vert^{N+1})$. Introducing $u=q+ip, v=q-ip$ produces equations of the form \eqref{uvk3} upon truncation of the remainder. However, polynomial transformations and more specifically remainder terms of the form $\mathcal O(\vert (r_3^2,\epsilon_3)\vert^{N+1})$, obtained by this standard theory, are not adequate for our purposes. Indeed, whereas these remainder terms can be made as small as desired by restricting the domain further and by increasing $N$, they are not uniformly small with respect to the small parameter $\epsilon=r_3^3\epsilon_3$. For this we need remainder terms that vanish (like those of the form $\mathcal O(r_3^{2N}\epsilon_3^N)$ in \eqref{uvk3}) along jets with respect to both $r_3$ and $\epsilon_3$. 
 \end{remark}

 \begin{remark}
  \response{In the smooth case $\mu\in C^\infty$, it is a consequence of the previous lemma and the blowup method, that $f=f(r_3^2,\epsilon_3)$ has a \textit{formal} expansion in $\epsilon_3$ of the form
  \begin{align*}
   f(r_3^2,\epsilon_3) = \sum_{m=0} R_m(r_3^2)\epsilon_3^m,
  \end{align*}
  with each $R_m$ smooth. Setting $\tilde f(t,\epsilon)=f(r_3^2,\epsilon_3)$, dropping the tilde and inserting the formal series into \eqref{nuexpr0} gives the formal expansion of $\nu$
  \begin{align}
   \nu(t,\epsilon) = \sum_{m=0} \lambda_3(t) R_m(t)t^{\frac{1-3m}{2}} \epsilon^m-\frac12 \epsilon t^{-1} (1+2t \lambda_3(t)^{-1}\lambda_3'(t)) \eqlab{nuexpansion}
  \end{align}
  for $\lambda(t)=\sqrt{t} \lambda_3(t)$, $\lambda_3(0)=i$, 
using $t=r_3^2$, $\epsilon=r_3^3\epsilon_3$. In other words, the blowup method tells us directly how the terms $\tilde \nu_m(t)$ of the formal expansion  $\nu(t,\epsilon)=\sum_{m=0}^\infty \tilde \nu_m(t) \epsilon^m$ (which is unique) depend upon $t$ in the limit $t\rightarrow 0^+$. 

In particular by \eqref{nuexpansion} we have  
\begin{align*}
\tilde \nu_0(t) = t^{\frac12} \lambda_3(t) =\lambda(t),\quad \tilde \nu_1(t) = \lambda_3(t)R_1(t) t^{-\frac12}-\frac12 t^{-1}(1+2t \lambda_3(t)^{-1}\lambda_3'(t)),
\end{align*} and  
\begin{align*}
\tilde \nu_m(t)=\lambda_3(t) R_m(t)t^{\frac{1-3m}{2}}\quad \mbox{for all} \quad m \ge 2.
\end{align*}
Then by expanding $\lambda_3(t) R_m(t)$ into formal power series of $t$, we can write each $\tilde \nu_m$ as a Laurent series $\tilde \nu_m(t)=\sum_{l=0}^\infty \tilde \nu_{m,l} t^{\frac{1-3m+2l}{2}}$ in $\sqrt{t}$ with a finite principle part:
\begin{align*}
 P(\tilde \nu)_m(t) :=\sum_{l=0}^{\lfloor \frac{3m-2}{2}\rfloor } \tilde \nu_{m,l} t^{\frac{1-3m+2l}{2}},
\end{align*}
i.e. there are only finitely many terms with negative exponents $t^{-\frac{q}{2}}$, $q\in \mathbb N$.

We will need the following in our proof of \thmref{mainTP2} later on:
\begin{lemma}\lemmalab{principle}
The principle part $P(\tilde \nu)_m$ of the Laurent series expansion of $\tilde \nu_m$ in $\sqrt{t}$ includes a nonzero $t^{-1}$-term if and only if $m=1$. 
\end{lemma}
\begin{proof}
There is $t^{-1}$-term in $P(\tilde \nu)_m$ only if $$1-3m+2l=-2.$$ If $m$ is even, then the left hand side of this equation is an odd number while the right hand side is even. We therefore consider $m=2n+1$ odd and use \eqref{evenodd2}:
\begin{align}
 \sum_{n=0}^\infty \int \tilde \nu_{2n+1}(t) dt\epsilon^{2n+1} &= -\frac{\epsilon}{2}\log \vert \sum_{n=0}^\infty \tilde \nu_{2n}(t) \epsilon^{2n} dt\vert\nonumber\\
 &=-\frac{\epsilon}{2}\log \vert \lambda(t)\vert-\frac{\epsilon}{2} \log\vert 1+\sum_{n=1}^\infty\lambda(t)^{-1} \tilde \nu_{2n}(t)\epsilon^{2n} dt\vert.\eqlab{expansionevenodd}
\end{align}
The Laurent expansion of $\tilde \nu_{2n+1}(t)$ in $t^{\frac12}$ contains a nonzero $t^{-1}$-term if and only if $\int \tilde \nu_{2n+1}(t) dt$ has a nonzero $\log(t)$-term in its expansion. For $n=0$, we have
\begin{align*}
 \int \tilde \nu_{1}(t) dt = -\frac12 \log \vert \lambda(t)\vert=-\frac14 \log t -\frac12 \log \vert \lambda_3(t)\vert,
\end{align*}
and $\tilde \nu_1$ therefore includes a nonzero $t^{-1}$-term in its Laurent expansion. For $n \ge 1$, there are no $\log$-terms on the right hand side of \eqref{expansionevenodd}. Indeed,
since $\tilde \nu_{2n}(t)$ is a Laurent series in $t^{\frac12}$, the last term in \eqref{expansionevenodd} is of the form $\sum_{n=1}^\infty L_{n}(t)\epsilon_{2n+1}$ with each $L_n$ being a Laurent series in $t^\frac{1}{2}$; $L_n$ can be expressed in terms of $\tilde \nu_0,\tilde \nu_2,\ldots,\tilde \nu_{2n}$ and $\lambda$, but the details are not important.

\end{proof}
 

}
 \end{remark}

 As in the chart $\{\bar t=-1\}$ and the expansion of $h_u$, recall \eqref{hu}, it will also be useful to expand $\nu_N$ in a similar way:
\begin{equation}\eqlab{nuNpmform}
\begin{aligned}
 \nu_{N}(r_3^2,\epsilon_3) &=T_{1}(\epsilon_3) +T_2(r_3^2)+\epsilon_3T_3(r_3^2) + r_3^2\epsilon_1^2T_4(r_3^2,\epsilon_3).
\end{aligned}
\end{equation}
where
\begin{align}
 T_{1}(\epsilon_3)&:=\nu_{N}(0,\epsilon_3)=i B_0(\epsilon_3) -\frac12\epsilon_3=i-\frac14 \epsilon_3+\mathcal O(\epsilon_3^2),\eqlab{T1expr}\\
 T_{2}(r_3^2)&:=\nu_N(r_1,0)-\nu_N(0,0) = \lambda_3(r_3^2)-i,\nonumber\\
 T_{3}(r_3^2) &:=\frac{\partial}{\partial \epsilon}\nu_N(r_3^2,0)-\frac{\partial}{\partial \epsilon}\nu_N(0,0)= -\frac12 r_3^2\lambda_3(r_3^2)^{-1} \lambda_3'(r_3^2),\nonumber
\end{align}
with $T_2(r_3^2),T_3(r_3^2)=\mathcal O(r_3^2)$. With $T_1,T_2,T_3$ defined by these expressions, the expansion \eqref{nuNpmform} follows from a Taylor-expansion of the difference  $$\nu_{N}(r_3^2,\epsilon_3) -\left(T_{1}(\epsilon_3) +T_2(r_3^2)+\epsilon_3T_3(r_3^2)\right).$$ Specifically, 
since $\nu_N\in C^{k-N-2}$ whenever $k\ge 2N+3$, it follows that $T_4$ is $C^{k-N-3}$-smooth, being in fact Gevrey-1 with respect to $\epsilon_3$ uniformly in $r_3^2\ge 0$. 

Consider $r_3=0$:
\begin{equation}\eqlab{k1eqnr10}
\begin{aligned}
 \dot x &= y_3,\\
 \dot y_3 &= -x-\frac12 \epsilon_3 y_3,\\
 \dot \epsilon_3 &= -\frac32 \epsilon_3^2.
\end{aligned}
\end{equation}
For this invariant subsystem, \lemmaref{k3transformation} shows that the transformation $(x,y_3,\epsilon_3)\mapsto (u,v,\epsilon_3)$ defined by 
\begin{align}\eqlab{diagr0}
 \begin{pmatrix}
  x\\
  y_3
 \end{pmatrix}&=\begin{pmatrix}
 B_0(\epsilon_3)&\overline B_0(\epsilon_3)\\
 i & -i\end{pmatrix}
\begin{pmatrix}
  u\\
  v
 \end{pmatrix},
\end{align}
diagonalizes this subsystem:
\begin{equation}\eqlab{uvr30}
\begin{aligned}
 \dot u &=T_1(\epsilon_3) u,\\
 \dot v &=\overline T_1(\epsilon_3)v.
\end{aligned}
\end{equation}
 In the following, let
 \begin{align*}
  A_1(\epsilon_3):=\exp\left(-\frac23 \int \epsilon_3^{-2} T_1(\epsilon_3)d\epsilon_3\right).
 \end{align*}
 $u=A_1(\epsilon_3)$, $v=0$ is then a solution of \eqref{uvr30}.
Using \eqref{diagr0}, we therefore obtain solutions $x_{\mathbb C}$ and $y_{3,\mathbb C}$ of \eqref{k1eqnr10} as follows  
 \begin{align}
  x_{\mathbb C}(\epsilon_3):=B_0(\epsilon_3) A_1(\epsilon_3), \eqlab{x1c}
 \end{align}
 and 
 \begin{align}
  y_{3,\mathbb C}(\epsilon_3):=-\frac32 \epsilon_3^2 x_{\mathbb C}'(\epsilon_3)=i A_1(\epsilon_3),\eqlab{y1c}
 \end{align}
  using \eqref{uv3} with $\lambda_3(0)=i$ in the last equality.  By invoking the expansion of $T_1$ in \eqref{T1expr}, we shall fix the integration defining $A_1$ such that 
  \begin{align}
   A_1(\epsilon_3) &=\epsilon_3^{1/6}\exp\left(i\left(\frac{2}{3} \epsilon_3^{-1}-\frac{\pi}{4}\right)\right)\left(1+\mathcal O(\epsilon_3)\right).\eqlab{A1Casymp}
  \end{align}
%
\response{
\begin{lemma}\lemmalab{thisA1}The remainder term of \eqref{A1Casymp} is Gevrey-1 in $\epsilon_3$.
 \end{lemma}
 \begin{proof} The factor $1+\mathcal O(\epsilon_3)$ in \eqref{A1Casymp} can be written  as
\begin{align*}
 \exp\left(-\frac23 \int_0^{\epsilon_3} T_{1,2}(s) ds\right),
\end{align*}
where $\epsilon_3^2 T_{1,2}(\epsilon_3)$ denotes the Gevrey-1 smooth remainder function in \eqref{T1expr}. 
\end{proof}} 
From \eqref{A1Casymp}, we obtain the following expansion
\begin{align}
 x_{\mathbb C}(\epsilon_3)= \epsilon_3^{1/6}\exp\left(i\left(\frac{2}{3} \epsilon_3^{-1}-\frac{\pi}{4}\right)\right)\left(1+\mathcal O(\epsilon_3)\right),\eqlab{x1Casymp}
 \end{align}
 \response{with remainder terms that are also smooth in $\epsilon_3$.}
  The asymptotics \eqref{x1Casymp} correspond to a unique solution, which can be related to the Airy-functions $\text{Ai}$ and $\text{Bi}$. We expect that this is known among experts (see e.g. \cite{costin1996a}), but give a direct proof, only based upon \lemmaref{center}.
 \begin{lemma}\lemmalab{xC}
 $(x_{\mathbb C}(\epsilon_3),y_{3,\mathbb C}(\epsilon_3))$ solves \eqref{k1eqnr10} upon eliminating time and by linearity so do $(\textnormal{Re}\,x_{\mathbb C}(\epsilon_3),\textnormal{Re}\,y_{3,\mathbb C}(\epsilon_3))$ and $(\textnormal{Im}\,x_{\mathbb C}(\epsilon_3),\textnormal{Im}\,y_{3,\mathbb C}(\epsilon_3))$.
%
In particular,
 \begin{align}
  x_{\mathbb C}(\epsilon_3)&=\sqrt{\pi}(\operatorname{Ai}-i\operatorname{Bi})(-\epsilon_3^{-2/3}),\eqlab{AirRexC}\\
  y_{3,\mathbb C}(\epsilon_3)&=-\sqrt{\pi}\epsilon_3^{1/3} (\operatorname{Ai}-i\operatorname{Bi})'(-\epsilon_3^{-2/3}).\eqlab{AirReyC}
 \end{align}
\end{lemma}
\begin{proof}
The first statements of the lemma are trivial. In the following we therefore only prove \eqref{AirRexC}. Notice that \eqref{AirReyC} follows the first equality in \eqref{y1c}.

From \eqref{x1Casymp}, using \eqref{Aitau_} and \eqref{Bitau_} together with \eqref{cc23TP}, it follows that the equality \eqref{AirRexC} holds up terms of order $\epsilon_3^{7/6}$. Write $$\epsilon_3^{7/6}X(\epsilon_3):=x_{\mathbb C}(\epsilon_3)-\sqrt{\pi}\left(\operatorname{Ai}(-\epsilon_3^{-2/3})-i\operatorname{Bi}(-\epsilon_3^{-2/3})\right). $$ Then it follows that $X$ is bounded and by linearity $\epsilon_3^{7/6}X(\epsilon_3)$ is also a solution of \eqref{k1eqnr10} when written as a second order equation for $x(\epsilon_3)$. Consequently, we obtain the following equations for $X$ and $Y:=-\frac32 \epsilon_3^2X'-\frac74 \epsilon_3 X$:
\begin{align*}
 \dot{X} &= Y + \frac74 \epsilon_3 X,\\
 \dot{Y} &=-X+\frac54 \epsilon_3 Y,\\
 \dot \epsilon_3 &=-\frac32 \epsilon_3^2.
\end{align*}
We proceed as in the proof of \lemmaref{k3transformation} and seek a diagonalization of this system by the following form:
\begin{align}
 \begin{pmatrix}
  X\\
  Y
 \end{pmatrix} =\begin{pmatrix}
 F(\epsilon_3) &\overline F(\epsilon_3)\\
 i & - i
 \end{pmatrix}\begin{pmatrix}
 U\\
 V\end{pmatrix},\eqlab{XYUV}
\end{align}
with $F(0)=1$. The proof follows \lemmaref{k3transformation} so we leave out further details but obtain 
\begin{align*}
 \dot U &=\left(i+\frac32 \epsilon_3+\mathcal O(\epsilon_3^2)\right)U,\\
 \dot V &=\left(-i+\frac32 \epsilon_3+\mathcal O(\epsilon_3^2)\right)V.
\end{align*}
Therefore 
\begin{align*}
 \frac{d}{d\epsilon_3} UV = -2\epsilon_3^{-1} \left(1+\mathcal O(\epsilon_3)\right) UV.
\end{align*}
Consequently, $UV$ is bounded for $\epsilon_3\rightarrow 0$ iff $UV\equiv 0$ iff $X\equiv 0$, completing the proof. 
\end{proof}

We now integrate the $u,v$-system \eqref{uvk3} using \eqref{nuNpmform}, \lemmaref{xC} and initial conditions on $\epsilon_3=\delta$. We will use $r_3$ as an independent variable and therefore consider
\begin{equation}\eqlab{uvr3}
\begin{aligned}
 \frac{du}{dr_3} &=\frac{2}{r_3 \epsilon_3(r_3)} \nu_N(r_3,\epsilon_3(r_3)) u +\mathcal O(r_3^{2N+1}\epsilon_3(r_3)^N)v,\\
 \frac{dv}{dr_3} &=\mathcal O(r_3^{2N+1}\epsilon_3(r_3)^N)u+\frac{2}{r_3 \epsilon_3(r_3)} \overline{\nu}_N(r_3,\epsilon_3(r_3)) v.
 \end{aligned}
 \end{equation}
 obtained by dividing the equations for $u$, $v$ by $\dot r_3$. 
%
Here from the conservation of $\epsilon=r_3^3\epsilon_3$ we have introduced
\begin{align}
 \epsilon_3(r_3):=r_3^{-3}\epsilon.\eqlab{eps3r3}
\end{align}
%
\begin{lemma}\lemmalab{P3Mapping}
Fix $\nu>0$ small enough. For all $\epsilon>0$ small enough, let $r_{3,\textnormal{in}}=\epsilon^{1/3}\delta^{-1/3}$ and define
 \begin{align}
   F_3(r_3):=&\sqrt[4]{\frac{\widehat \mu(r_{3,\textnormal{in}}^2)}{\widehat \mu(r_3^2)}}\exp\left(i\frac{1}{\epsilon}\int_{r_{3,\textnormal{in}}^2}^{r_3^2} \left(\sqrt{\mu(s^2)}-s\right)ds^2\right)\\
   &\times \exp\left(\int_{r_{3,\textnormal{in}}^2}^{r_3^2} s^{-3} \response{\epsilon} T_4(s^2,s^{-3} \epsilon) ds^2\right).\eqlab{F3expr0}
 \end{align}
Finally, let $u(r_{3,\textnormal{in}})=u_{in}$, $v(r_{3,\textnormal{in}})=v_{in}$ denote an initial condition. Then:

%
 
%
%
 

\begin{enumerate}
\item \label{1uexpansion} The corresponding solution $(u(r_3),v(r_3))$ of \eqref{uvr3} with $(u(r_{3,\textnormal{in}}),v(r_{3,\textnormal{in}})=(u_{\textnormal{in}},v_{\textnormal{in}})$ satisfies
\begin{equation}\eqlab{uvout}
 \begin{aligned}
  u(r_3) &= \left(A_1(\epsilon_{3}(r_3))F_3(r_3)A_1(\delta)^{-1} +\mathcal O(r_{3,\text{in}}^{2N})\right)\left(u_{in}+\mathcal O(r_{3,\text{in}}^{2N})v_{in}\right),\\
  v(r_3)&=\left(\overline{A}_1(\epsilon_3(r_3)) \overline F_3(r_3)\overline{A}_1(\delta)^{-1} +\mathcal O(r_{3,\text{in}}^{2N})\right) \left(v_{in}+\mathcal O(r_{3,\text{in}}^{2N})u_{in}\right),
 \end{aligned}
 \end{equation}
for $r_3\in [r_{3,\textnormal{in}},\nu]$. If $v_{in}=\overline u_{in}$ then $v(r_3) = \overline u(r_3)$ for all $r_3\in [r_{3,\textnormal{in}},\nu]$. 
%
\item \label{1Fexpansion} $F_3(r_3)$ has the following expansion: 
\begin{align}
 F_3(r_3) = {\widehat \mu(r_3^2)}^{-1/4}\exp\left(i\frac{1}{\epsilon }\int_{0}^{r_3^2} \left(\sqrt{\mu(s^2)}-s\right)ds^2\right)\left(1+\mathcal O(r_{3,\textnormal{in}}^2)\right).\eqlab{E3expr}
\end{align}
as $r_{3,\textnormal{in}}\rightarrow 0$. 
\end{enumerate}
\response{Assume next, that $t\mapsto \mu(t,E)$, with $\mu(0,E)=0, \frac{\partial \mu}{\partial t}(0,E) = 1$ for all $E\in D$, is $C^\infty$, jointly in $t$ and $E$. Consider $r_3=\nu>0$ fixed small enough. }
\begin{enumerate}[resume*]
\item \label{Fremaindersmoothness} \response{Let $L=2\lfloor \frac{N}{3}\rfloor+1$. Then the $\mathcal O(r_{3,\textnormal{in}}^2)$-remainder term in the expansion of $F_3$ in \eqref{E3expr} is of the form $r_{3,\textnormal{in}}^2\overline F_3(r_{3,\textnormal{in}},E)$ with $\overline F_3\in C^{L}$. }
\item \label{2remaindersmoothness}
\response{ Let $M=\lfloor \frac{N}{2}\rfloor-1$. Then for fixed $r_3>0$, each $\mathcal O(r_{3,\textnormal{in}}^{2N})$-remainder term in \eqref{uvout} $C^{M}$-smooth with respect to and $r_{3,\textnormal{in}}$ and $E$, with the order of the remainder changing upon differentiation as follows:
\begin{align*}
 \frac{\partial^{i+j}}{\partial r_{3,\textnormal{in}}^{i}\partial E^j}  \mathcal O(r_{3,\textnormal{in}}^{2N}) = \mathcal O(r_{3,\textnormal{in}}^{2N-4i-3j}), 
\end{align*}
for all $0\le i+j \le M$.}
\end{enumerate}

\end{lemma}
\begin{proof}
We first solve the truncated system and then define a transformation based upon this: Consider therefore
\begin{align}
 Q(r_3):= \exp\left(\int_{r_{3,\textnormal{in}}}^{r_3} \frac{2}{r_3 \epsilon_3(r_3)}\nu_N(r_3,\epsilon_3(r_3)) dr_3\right),\eqlab{Q3}
\end{align}
and define
$(\tilde u,\tilde v)(r_3)$ by
\begin{align*}
 u(r_3) &= Q(r_3)\tilde u(r_3),\quad v(r_3) =\overline Q(r_3) \tilde v(r_3). 
\end{align*}
Then
\begin{equation}\eqlab{tildeuv}
\begin{aligned}
 \frac{d\tilde u}{dr_3}&=\mathcal O(r_{3,\text{in}}^{2N}) Q(r_3)^{-1}\overline Q(r_3) \tilde v,\\
 \frac{d\tilde v}{dr_3}&=\mathcal O(r_{3,\text{in}}^{2N}) \overline Q(r_3)^{-1}Q(r_3) \tilde u.
\end{aligned}
\end{equation}
Since $\vert \overline Q(r_3)^{-1}Q(r_3)\vert =1$, it is straightforward to integrate these equations from $r_{3,\text{in}}$ to $r_3$ and estimate
\begin{align*}
 \tilde u(r_3) &= (1+\mathcal O(r_{3,\text{in}}^{2N}))\tilde u(r_{3,\text{in}})+\mathcal O(r_{3,\text{in}}^{2N})\tilde v(r_{3,\text{in}}),\\
 \tilde v(r_3) &= (1+\mathcal O(r_{3,\text{in}}^{2N}))\tilde u(r_{3,\text{in}})+\mathcal O(r_{3,\text{in}}^{2N})\tilde v(r_{3,\text{in}}),
\end{align*}
for all $r_3\in [r_{3,\textnormal{in}},\nu]$.
To complete the proof of item (\ref{1uexpansion}), the only thing left is therefore to expand $Q$ \eqref{Q3}. Using \eqref{nuNpmform}, we obtain
\begin{align}
Q(r_3)= &\exp\left(\int_{r_{3,\textnormal{in}}}^{r_3} \frac{2}{r_3\epsilon_3(s)}T_1(\epsilon_3(s)) ds\right)\times \nonumber \\
&\exp\left(i\int_{r_{3,\textnormal{in}}}^{r_3} \frac{2}{s \epsilon_3(s)}\left(\sqrt{\widehat \mu(s^2)}-1\right)ds\right)\exp\left(-\int_{r_{3,\textnormal{in}}}^{r_3} \frac{s\widehat \mu'(s^2)}{2\widehat \mu(s^2)}ds\right)\times \nonumber \\
&\exp\left(\int_{r_{3,\textnormal{in}}}^{r_3}  2s^{-2} \epsilon T_4(s^2,s^{-3}\epsilon) ds\right) \nonumber \\
=&\exp\left(-\frac23 \int_{\delta}^{\epsilon_{3}(r_3)} \frac{1}{s^2}T_1(s)ds\right)\times \eqlab{Qexpr0}\\
&\exp\left(i\frac{1}{\epsilon}\int_{r_{3,\textnormal{in}}^2}^{r_3^2} \left(\sqrt{\mu(s^2)}-s\right)ds^2\right)\exp\left(-\int_{r_{3,\textnormal{in}}^2}^{r_3^2} \frac{\widehat \mu'(s^2)}{4\widehat \mu(s^2)}ds^2\right)\times \eqlab{Qexpr}\\
&\exp\left(\int_{r_{3,\textnormal{in}}^2}^{r_3^2} s^{-3} \epsilon T_4(s^2,s^{-3}\epsilon) ds^2\right)\nonumber\\
=&A_1(\epsilon_3(r_3))A_1(\delta)^{-1} F_3(r_{3,\textnormal{in}}).\nonumber
\end{align}
%
%
We recall that $\epsilon = r_{3,\textnormal{in}}^3 \delta$.

\response{For the proof of item (\ref{1Fexpansion}), we notice from \eqref{F3expr0} that the factor $1+\mathcal O(r_{3,\textnormal{in}}^2)$  in \eqref{E3expr} is given by 
\begin{align}
 \sqrt[4]{\widehat \mu(r_{3,\textnormal{in}}^2)} \exp\left(-i\frac{1}{\epsilon}\int_0^{r_{3,\textnormal{in}}^2}  \left(\sqrt{\mu(s^2)}-s\right)ds^2 \right) \exp\left(\int_{r_{3,\textnormal{in}}^2}^{r_3^2} s^{-3} \epsilon T_4(s^2,s^{-3} \epsilon) ds^2\right).\eqlab{thisexpansion}
\end{align}
To prove \eqref{E3expr} we show that each factor is of the form $1+\mathcal O(r_{3,\textnormal{in}}^2)$. The first factor is obvious. (It is even $C^{k-1}$-smooth as a function of $r_{3,\textnormal{in}}^2$ since $\widehat \mu$ is so.) For the second factor, we write
\begin{align*}
 \sqrt{\mu(s^2)}-s = s\left(\sqrt{\widehat \mu(s^2)} - 1\right):=s^3 \overline \mu(s^2),
\end{align*}
for some $\overline \mu$, seeing that $\widehat \mu(0)=1$. If $\mu$ is $C^k$ then $\overline \mu$ is $C^{k-2}$. We have 
\begin{align}
\exp\left(-i\frac{1}{\epsilon}\int_0^{r_{3,\textnormal{in}}^2}  \left(\sqrt{\mu(s^2)}-s\right)ds^2 \right) = \exp\left(-i{r_{3,\textnormal{in}}^2}{\delta^{-1}}\int_0^{1}  s^3 \overline \mu(r_{3,\textnormal{in}}^2 s^2) ds^2 \right).\eqlab{2ndterm}
\end{align}
This expression gives the desired form $1+\mathcal O({r_{3,\textnormal{in}}^2})$. (In fact, by Leibniz rule of differentiation, \eqref{2ndterm} is $C^{k-2}$-smooth as a function of $r_{3,\text{in}}^2$ whenever $\mu\in C^k$.) 
  Finally, the bound $1+\mathcal O(r_{3,\textnormal{in}}^2)$ of the third factor in \eqref{thisexpansion} follows directly from
 \begin{align*}
  \int_{r_{3,\textnormal{in}}^2}^{r_3^2} s^{-3} \mathcal O(r_{3,\textnormal{in}}^3) ds^2 = \mathcal O(r_{3,\textnormal{in}}^2),
 \end{align*}
 upon using a uniform bound on $T_4$. 
 This completes the proof of  \eqref{E3expr} in item (\ref{1Fexpansion}).
 
 We now turn to the proof of item (\ref{Fremaindersmoothness}), assuming that $\mu=\mu(t,E)$ is $C^\infty$. We prove this by showing that each of the factors of \eqref{thisexpansion} with $r_3=\nu>0$ fixed are smooth functions of $r_{3,\textnormal{in}}$ and $E$, specifying the degree of smoothness for each factor. The first factor is clearly $C^\infty$, see discussion above. To see that the second factor is also $C^\infty$, we use \eqref{2ndterm} which still holds under the assumptions on $\mu$ with $\overline \mu(t,E)$ being $C^\infty$. We now turn to the third factor. Notice first that $T_4$ is $C^\infty$ under the assumptions, jointly in $r_3^2,\epsilon_3$ and $E$.  
By scaling $r_{3,\textnormal{in}}$, $s^2$, and redefined $T_4$ based upon these scalings, we achieve $r_{3,\textnormal{in}}=r$, $\nu=1$ and $\delta=1$ and write the relevant integral in the following form:
\begin{align}
\int_{r^2}^{1} s^{-3} r^3 T_4(s^2,s^{-3}r^3) ds^2.\eqlab{T4int}
\end{align}
The dependency on $E$ is regular, and we will therefore obtain the desired result when we have demonstrated the smoothness properties with respect to $r$. (In fact, we can replace $T_4$ with $\frac{\partial^j T_4}{\partial E^j}$ without any changes to our arguments.) For this reason, $E$ is therefore suppressed in the following. $T_4$ remains $C^{\infty}$, now on the domain $[0,1]^2$. 
 
 The full details of the proof is delayed to \appref{T4}. Here we only present a proof under the (simplifying) assumption that $T_4$ is in fact analytic $T_4(r_3^2,\epsilon_3) = \sum_{m,n} T_{4mn} r_3^{2m}\epsilon_3^n$. (The proof is simpler in this case and it provides valuable insight into the actual proof for the $C^\infty$-case in \appref{T4}.) Then \eqref{T4int} becomes
 \begin{align}
  \sum_{m,n} T_{4mn}  r^{3(n+1)}& \int_{r^2}^1 s^{2m-3(n+1)} ds^2\nonumber\\
  &=\sum_{m,n} T_{4mn}  r^{3(n+1)} \left[\frac{2}{2m-3n-1} s^{2m-3n-1} \right]_{s^2=r^2}^1\nonumber\\
  &=\sum_{m,n} \frac{2T_{4mn}}{{2m-3n-1}} \left(r^{3(n+1)}-r^{2m+2}\right),\eqlab{T4analyticcase}
 \end{align}
provided that
\begin{align}
 T_{4mn}=0 \quad \text{whenever}\quad {2m-3n-1}=0,\eqlab{condres}
\end{align}
for all $m,n\in \mathbb N$. If \eqref{condres} holds, then, under the simplifying assumption that $T_4$ is analytic on $[0,1]^2$, we have that the series $4\sum_{n,m}\vert T_{4nm}\vert$ is a convergent majorant series of the right hand side of \eqref{T4analyticcase}. 

We now proceed to study \eqref{condres}. If $n$ is even then ${2m-3n-1}$ is odd. We therefore deduce that $n$ is odd, and in turn that ${2m-3n-1}= 0$, $m,n\in \mathbb N$, if and only if
\begin{align*}
 m=3l-1,\quad n=2l-1,
\end{align*}
with $l\in \mathbb N$. 
The corresponding monomials of $T_4$ are then 
\begin{align}
 T_{4(3l-1)(2l-1)}r_3^{2(3l-1)}\epsilon_3^{2l-1} = T_{4(3l-1)(2l-1)}r_3^{-2} \epsilon_3^{-1} \epsilon^{2l},\eqlab{monomialsT4}
\end{align}
with $\epsilon = r_3^3\epsilon_3$. Moreover, $T_{4(3l-1)(2l-1)}=0$ is equivalent to
\begin{align}
 \frac{\partial^{5l-2} }{\partial r_3^{2(3l-1)} \partial \epsilon_3^{2l-1}}T_4(0,0)=0,\quad \mbox{for}\quad l\in \mathbb N,\eqlab{condres2}
\end{align}
\begin{lemma}\lemmalab{T4prop}
 Fix $N\in \mathbb N$ as in \lemmaref{k3transformation}. Then \eqref{condres2} holds true for all $1\le l\le \lfloor \frac{N}{3}\rfloor$ and all $E$. 
\end{lemma}
\begin{proof}
We first show that \eqref{condres} holds true in the formal limit $N\rightarrow \infty$. 
Inserting \eqref{monomialsT4} into the equation for $\nu_N$ in \eqref{nuNpmform} (in the formal limit $N\rightarrow \infty$), gives a factor of $r_3^2\epsilon_3^2$ and therefore corresponding monomials of the order
\begin{align*}
 T_{4(3l-1)(2l-1)} \epsilon_3 \epsilon^{2l}.
\end{align*}
We go from $\nu_N$ in \eqref{nuNpmform} (again in the formal limit $N\rightarrow \infty$) to $\nu$ in \eqref{nuexpr0} by multiplying by $r_3=\sqrt{t}$.
Consequently, we deduce that monomials of the form \eqref{monomialsT4} give rise to 
\begin{align*}
 T_{4(3l-1)(2l-1)} r_3 \epsilon_3 \epsilon^{2l} = T_{4(3l-1)(2l-1)} r_3^{-2} \epsilon^{2l+1}=T_{4(3l-1)(2l-1)} t^{-1} \epsilon^{2l+1}
\end{align*}
in the expansion of $\nu$, recall \eqref{nuexpansion}. But \lemmaref{principle} shows that such terms are absent unless $2l+1=1$, i.e. $l=0$. Consequently, $T_{4(3l-1)(2l-1)}=0$ for all $l\in \mathbb N$ in the formal limit $N\rightarrow \infty$.

For fixed $N$, we have by construction that $\nu_N$ agrees with the formal limit (at the level of the partial derivatives) up to (but not including) terms of order $\mathcal O(r_3^{2(N+1)}\epsilon_3^{N+1})$. In turn, $T_4$ then agrees with the formal limit (again at the level of the partial derivatives) up to (but not including) terms of the order $\mathcal O(r_3^{2(N)}\epsilon_3^{N-1})$. It follows that \eqref{condres2} holds for all $l\in \mathbb N$ such that $3l-1\le N-1$ and $2l-1\le N-1$. This gives the desired result.
\end{proof}}
%
%
%
%
\response{Finally, regarding the statement in item (\ref{2remaindersmoothness}), we first differentiate $Q$. Since $\mu(0,E)=0$ for all $E$, every factor of $Q$ is regular with respect to $r_{3,\text{in}}\rightarrow 0$, except for \eqref{Qexpr0} and the first factor in \eqref{Qexpr}. Here we also use item (\ref{Fremaindersmoothness}) and \lemmaref{thisA1}. We therefore focus on these terms. Consider first differentiation with respect to $E$. Since \eqref{Qexpr0} does not depend upon $E$, we just consider the second factor:
  \begin{align}
   \mathcal O(1)r_{3,\text{in}}^{-3} \int_{r_{3,\text{in}}^2}^{\nu} \frac{\mu'_E(s^2,E)}{s\sqrt{\widehat \mu(s^2,E)}} ds^2.\eqlab{Qprime}
  \end{align}
$\int_0^1 x^{-1/2} dx = 2$ is well-defined, and the integral in \eqref{Qprime} is therefore bounded as $r_{3,\text{in}}\rightarrow 0$. Hence, we can bound \eqref{Qprime} and consequently all of $Q'_E$ by $Cr_{3,\text{in}}^{-3}$ for some constant $C>0$. In turn, upon differentiating \eqref{tildeuv} with respect to $E$, it follows that
\begin{align*}
 \frac{d(\tilde u'_E)}{dr_3 } &=\mathcal O(r_{3,\text{in}}^{2N-3}),\quad \frac{d(\tilde v'_E)}{dr_3 } =\mathcal O(r_{3,\text{in}}^{2N-3}).
\end{align*}
Differentiation with respect to $r_{3,\text{in}}$ proceeds analogously, but now we bound $Q'_{r_{3,\textnormal{in}}}$ by $Cr_{3,\text{in}}^{-4}$ due to \eqref{Qexpr0} and \eqref{Qexpr}. Higher order derivatives $\frac{\partial^{i+j}}{\partial r_{3,\text{in}}^i \partial E^j}$ can be handled similarly, provided that $2N-4i-3j>0$. A sufficient condition is that $i+j< \frac{N}{2}$. This completes the proof of item (\ref{2remaindersmoothness}).}
\end{proof}
We now use the preceeding lemma to extend the solution \eqref{Wr3} up to $r_3=\nu$. For this we use the value at $\epsilon_3=\delta$, corresponding to $r_3=r_{3,\textnormal{in}}=\epsilon^{1/3}\delta^{-1/3}$, as initial condition for the flow.
Using \lemmaref{xC}, we can write this as follows:
\begin{align*}
 x_{\textnormal{in}} &=\frac{1}{\sqrt{\pi}} \operatorname{Re}\left(B_0(\delta)A_1(\delta)\right)+\mathcal O(r_{3,\textnormal{in}}^2),\\
 y_{3,\textnormal{in}} &=\frac{1}{\sqrt{\pi}}\operatorname{Re}\left(iA_1(\delta)\right)+\mathcal O(r_{3,\textnormal{in}}^2),
\end{align*}
The remainder is smooth with respect to $r_{3,\textnormal{in}}^2$.
Upon application of the inverse of \eqref{uv3} with $\epsilon_3=\delta$, $r_3=r_{3,\textnormal{in}}$, we therefore obtain the following initial conditions on $u,v$:
 \begin{align*}
 u_{\textnormal{in}} &=\frac{1}{2\sqrt{\pi}} A_1(\delta)+\mathcal O(r_{3,\textnormal{in}}^2),\\
 v_{\textnormal{in}} &=\frac{1}{2\sqrt{\pi}} \overline{A}_1(\delta)+\mathcal O(r_{3,\textnormal{in}}^2).
\end{align*}
Inserting this into \eqref{uvout} produces 
 \begin{equation}\eqlab{uvout2}
 \begin{aligned}
  u(r_3) &=\frac{1}{2\sqrt{\pi}}A_1(\epsilon_{3}(r_3)) {\widehat \mu(r_3^2)}^{-1/4}e^{\frac{i}{\epsilon }\int_{0}^{r_3^2} (\sqrt{\mu(s^2)}-s)ds^2}(1+\mathcal O(r_{3,\textnormal{in}}^2)),\\
  v(r_3)&=\overline u(r_3),
  \end{aligned}
 \end{equation}
 upon using \lemmaref{P3Mapping}, see \eqref{E3expr}. We then obtain the following:
 \begin{lemma}\lemmalab{solchart3}
  The solution \eqref{Wr3} can be extended into the $\{\bar t=1\}$-chart upon using the values at $\epsilon_3=\delta$, $r_3=r_{3,\textnormal{in}}:=\epsilon^{1/3}\delta^{-1/3}$ as initial conditions for the flow of \eqref{eqnk3}. This gives
  \begin{align*}
  x(r_3) &={\widehat \mu(r_3^2)}^{-1/4}\textnormal{Re}\bigg((\operatorname{Ai}-i\operatorname{Bi})(-\epsilon_3^{-2/3})e^{\frac{i}{\epsilon }\int_{0}^{r_3^2} \left(\sqrt{\mu(s^2)}-s\right)ds^2}(1+\mathcal O(\epsilon^{2/3}))\bigg),\\
  y_3(r_3) &=-\epsilon_3^{1/3} {\widehat \mu(r_3^2)}^{1/4}\textnormal{Re}\bigg((\operatorname{Ai}-i\operatorname{Bi})'(-\epsilon_3^{-2/3})e^{\frac{i}{\epsilon }\int_{0}^{r_3^2} \left(\sqrt{\mu(s^2)}-s\right)ds^2}(1+\mathcal O(\epsilon^{2/3}))\bigg)
  \end{align*}
  for $r_3\in [r_{3,\textnormal{in}},\nu]$ with $u(r_3)$ given by \eqref{uvout2} and $\epsilon_3(r_3)=r_3^{-3}\epsilon$. 
   \end{lemma}
 \begin{proof}
By \lemmaref{k3transformation}, we have
\begin{align*}
 f(r_3^2,\epsilon_3) = B_0(\epsilon_3)(1 + \mathcal O(r_3^2\epsilon_3)).
\end{align*}
Consequently, upon applying  \eqref{uv3} to \eqref{uvout2} we have
\begin{equation}\eqlab{xy3proof}
 \begin{aligned}
  x(r_3) &=2\text{Re}(B_0(\epsilon_3)u(r_3)(1+\mathcal O(r_{3,\textnormal{in}}^2)),\\
  y_3(r_3) &=2\text{Re}(i \sqrt{\widehat \mu(r_3^2)} u(r_3)),
 \end{aligned}
 \end{equation}
The result then follows from \eqref{x1c}, \eqref{y1c}, \eqref{AirRexC} and \eqref{AirReyC}.
 \end{proof}
 \response{
 \begin{lemma}\lemmalab{smoothnessfinal}
  Suppose that $\mu=\mu(t,E)$ is $C^\infty$, fix $N\in \mathbb N$ (as in \lemmaref{k3transformation}) and $r_3=\nu>0$ small enough, and let $M=\lfloor \frac{N}{2}\rfloor-1$. Then the remainder terms of \eqref{xy3proof} are $C^M$ with respect to $\epsilon^{1/3}$ and $E$.  
   \end{lemma}
\begin{proof}
 The result follows from the smoothness of $f$ and \lemmaref{P3Mapping}, see items (\ref{Fremaindersmoothness}) and (\ref{2remaindersmoothness}). Recall here that $r_{3,\textnormal{in}}=\epsilon^{1/3} \delta^{-1/3}$ in \lemmaref{P3Mapping}. 
\end{proof}
}

\subsection{Completing the proof of \thmref{mainTP}}
The local expansions in \lemmaref{solchart1}, \lemmaref{solchart2} and \lemmaref{solchart3} prove \thmref{mainTP} item (\ref{expansionsol}) upon using \eqref{chartT_1TP}--\eqref{chartT1TP} and $t=-r_1^2$, $t=\epsilon^{2/3}t_2$, and $t=r_3^2$ for the reparametrization. 

For \thmref{mainTP} item (\ref{manifoldWu1}) we notice that the space is given by $(lx(\nu),ly(\nu))$, $l\in \mathbb R$, with $(x(t),y(t))$ given in item (\ref{expansionsol}). The result then follows from the expansion in item (\ref{expansion3}) upon setting $t=\nu$ and using the expansions of $\operatorname{Ai}$ and $\operatorname{Bi}$, see \eqref{Aitau_}, \eqref{Bitau_}, \eqref{AiPtau_} and \eqref{BiPtau_}. Equivalently, we can put $r_3=\sqrt{\nu}$ in \eqref{xy3proof} and use \eqref{A1Casymp} in \eqref{uvout2} with $\epsilon_3(\sqrt{\nu})= \nu^{-3/2} \epsilon$. This gives
\begin{equation}\eqlab{xy3final}
\begin{aligned}
 x &= \frac{1}{\sqrt{\pi}}\epsilon_3^{1/6}\widehat \mu (\nu)^{-1/4} \text{Re}\big(e^{i(\frac{2}{3} \epsilon_3^{-1}-\frac{\pi}{4})}e^{\frac{i}{\epsilon }\int_{0}^{\nu} \left(\sqrt{\mu(s^2)}-s\right)ds^2}(1+\mathcal O(\epsilon^{2/3}))\big),\\
 y_3 &= -\frac{1}{\sqrt{\pi}}\epsilon_3^{1/6}\widehat \mu (\nu)^{1/4} \text{Im}\big(e^{i(\frac{2}{3} \epsilon_3^{-1}-\frac{\pi}{4})}e^{\frac{i}{\epsilon }\int_{0}^{\nu} \left(\sqrt{\mu(s^2)}-s\right)ds^2}(1+\mathcal O(\epsilon^{2/3}))\big),
\end{aligned}
\end{equation}
since $\lambda_3(r_3^2)=i\sqrt{\widehat \mu(r_3^2)}$.
Upon simplifying, using $y=r_3y_3$, recall \eqref{chartT1TP}, we obtain the result. 

\subsection{Completing the proof of \thmref{mainTP2}}
\response{In the case where $\frac{\partial }{\partial t}\mu(0,E)=1$ for all $E\in D$, \thmref{mainTP2} follows from \eqref{xy3final} with $\epsilon_3=\nu^{-3/2}\epsilon$ and $\mu=\mu(t,E)$, \eqref{chartT1TP} and \lemmaref{smoothnessfinal}. In particular, upon using $\int_0^\nu s ds^2 = \frac23 \nu^{3/2}$, we obtain \eqref{WuExpr2} with
\begin{align*}
 X(\epsilon^{1/3},E) &= \operatorname{Re}\big[e^{-\frac{\pi}{4}i}e^{\frac{i}{\epsilon }\int_{0}^{\nu} \sqrt{\mu(s,E)}ds}(1+\epsilon^{2/3} Z_1(\epsilon^{1/3},E))\big],\\
 Y(\epsilon^{1/3},E) &= 
  \operatorname{Im}\big[e^{-\frac{\pi}{4}i}e^{\frac{i}{\epsilon }\int_{0}^{\nu} \sqrt{\mu(s,E)}ds}(1+\epsilon^{2/3} Z_2(\epsilon^{1/3},E))\big],
\end{align*}
for $Z_1,Z_2:[0,\epsilon_0^{1/3})\times D\rightarrow \mathbb C$ both $C^M$ smooth, upon taking $N$ large enough. Now writing each $$1+\epsilon^{2/3} Z_i(\epsilon^{1/3},E)=(1+\epsilon^{2/3}\rho_i(\epsilon^{1/3},E))e^{i\epsilon^{2/3}\phi_i(\epsilon^{1/3},E)},$$ $i=1,2$ in polar form, and simplifying, we obtain the result. The general case, $\frac{\partial }{\partial t}\mu(0,E)>0$, can be brought into the case $\frac{\partial }{\partial t}\mu(0,E)=1$ by a scaling of $t$ and $y$. Upon subsequently undoing this scaling in the expression for $W^u(\nu)$ we obtain the desired result for any $\mu(t,E)$ with $\mu(0,E)=0$, $\frac{\partial}{\partial t}\mu(0,E)>0$.}
\section{Discussion and future work}\seclab{discussion}

In this paper, we have presented a novel dynamical systems oriented approach to problems traditionally treated by WKB-methods or other methods based on formal or rigorous asymptotic expansions. In our main result, \thmref{mainTP}, we provide a description of solutions within $W^u$ in different domains that cover a full neighborhood of the turning point for all $0<\epsilon\ll 1$. 
In turn, we obtain a rigorous description of the transition of $W^u$ across $t=0$ for all $0<\epsilon\ll 1$, \response{see also our second main result \thmref{mainTP2}, which describes smoothness properties of $W^u$ in the case where $\mu$ is $C^\infty$, also in a parameter $E$.}

In a second paper \cite{kriszm22b}, we use this result to study the eigenvalue problem \eqref{eq0} for a potential $V$ with one minimum. In particular, we will provide a description of how the Bohr-Sommerfeld approximation \eqref{quant} can be understood. Notice that this is nontrivial, since for $E$ to be $\mathcal O(1)$, $n$ needs to be $\mathcal O(1/\epsilon)$ in \eqref{quant}. Moreover, we will show that \eqref{quant} in fact approximates all bounded eigenvalues uniformly, including the ``low-lying'' eigenvalues $E=\mathcal O(\epsilon)$ for $n=0,1,2\ldots$, in a certain sense which is made precise. 

In a third paper \cite{kriszm22c}, we then consider potentials with a least one local maximum. Here a different type of turning point has to be considered, corresponding to \eqref{secondordereq} with $\mu(0)=\mu'(0)=0$, $\mu''(0)< 0$. The results show that the eigenvalues in this case are not separated by $\mathcal O(\epsilon)$-distances, but lie closer than that, being instead separated by distances of order $\epsilon \log^{-1} \epsilon^{-1}$ and we describe this using the same techniques as in \cite{kriszm22b} and the present paper.

We emphasize that  tracking stable and unstable spaces, whose intersection correspond to eigenvalues of a differential operator,
is the central idea of the so-called Evans function method. This method has been used very successfully in the stability analysis of traveling wave solutions 
\cite{jones1990,kapitula2013a,sandstede2002} and also in the context of eigenvalues and resonances of the Schr\"{o}dinger equation \cite{kapitula2004}.

Finally, we emphasize that our approach can be used in other linear and nonlinear turning point problems with hyperbolic-elliptic  transitions,
e.g. Orr-Sommerfeld equation,  nonlinear (Hamiltonian) problems reducible to Painl\'{e}ve I or II, see e.g. \cite{haberman79,kristiansen2015a,lerman2016a}. In particular, 
\lemmaref{center}, see also \lemmaref{k3transformation}, and the whole analysis in the chart $\bar  t = 1$, provide a general template 
for hyperbolic to elliptic transitions. We demonstrate this in \cite{kriszm22c} in our study of \eqref{secondordereq} with $\mu(0)=\mu'(0)=0$, $\mu''(0)<0$. 
 We are therefore confident that the present paper will have an impact on these type of problems that is similar to the impact of the blow-up analysis of the singularly perturbed planar fold \cite{krupa_extending_2001}. This paper (i) explained the complicated structure of the classical expansions, which had been known for a long time, and (ii) 
triggered the development of a  host of new theoretical results and applications of GSPT during the last two decades.

\newpage 
\bibliography{refs}
\bibliographystyle{plain}
\newpage
\appendix 
\section{Proof of \lemmaref{center}}\applab{proof}
We consider \eqref{centerEq} under the assumption \eqref{centerAs}. For simplicity we consider the scalar case $n=1$ and write $\omega=\omega_1$; the generalization to any $n\in \mathbb N$ is straightforward but it is notationally slightly more involved. 

It is elementary to transform the system into the following form
\begin{align}
 x^2 \frac{dy}{dx} - i\omega y = f(x,y),\eqlab{conj}
\end{align}
with $f$ analytic and $f(0,y)=0$ for all $y$. We will solve this equation for $y=Y(x)$, $Y(0)=0$, on the space of analytic functions defined on a local sector centered along the positive real $x$-axis. We will do so by following \cite{bonckaert2008a} and (a) first transform the equation into an equation on the ``Borel-plane'' (through the Borel-transform $\mathcal B$), (b) apply a fixed-point argument there and then (c) obtain our desired solution by applying the Laplace transform. 

The Borel transform is defined in the following way: If $h(x)=\sum_{n=1}^\infty h_n x^n$ is a Gevrey-1 formal series: 
\begin{align*}
\vert h_n\vert \le a b^n n!,
\end{align*}
recall \eqref{gevrey1},
then the we define the Borel transform as
\begin{align*}
 \mathcal B(h)(u) = \sum_{n=0}^\infty \frac{h_{n+1}}{n!} u^n,
\end{align*}
which is analytic on $\vert u\vert <b^{-1}$. For the Laplace transform, on the other hand, we need analytic functions that are at most exponentially growing in a sector. With this in mind, define $S_\theta\subset \mathbb C$ as the sector centered along the positive real $x$-axis with opening $\theta\in (0,\pi)$, recall \eqref{Stheta}, and let $B(R)$ be the open ball of radius $R$ centered at $0$. Finally, set 
\begin{align*}
 \Omega := S_\theta\cup B(R).
\end{align*}
Then for any $\zeta>0$ 
we define the norm 
\begin{align}
\Vert \alpha\Vert_\zeta:=\sup_{u\in \Omega} \left\{\vert \alpha(u)\vert (1+\zeta^2 \vert u\vert^2)e^{-\zeta \vert u\vert}\right\},\eqlab{normG}
\end{align}
see \cite{bonckaert2008a}, 
on the space of analytic and exponentially growing functions on $\Omega$:
\begin{align*}
  \mathcal G:=\{\alpha:\alpha \mbox{  is analytic on  } \Omega \mbox{ and } \Vert \alpha\Vert_\zeta<\infty\}.
\end{align*}
The normed space $(\mathcal G,\Vert \cdot\Vert_\zeta)$ is a complete space. In the following we will take $\zeta$ large enough. For this purpose, it is important to note that the following obvious inequality holds:
\begin{align*}
 \Vert \alpha \Vert_{\zeta'}\le \Vert \alpha \Vert_{\zeta},
\end{align*}
for all $\zeta'\ge \zeta$ and all $\alpha\in \mathcal G$.
The factor $1+\zeta^2 \vert w\vert^2$ in the norm $\Vert \cdot \Vert_{\zeta}$ ensures that the convolution:
 \begin{align*}
  (\alpha \star \beta)(u) = \int_0^u \alpha(s) \beta(u-s)ds,
 \end{align*}
 is a continuous as a bilinear operator on $\mathcal G$. In particular, we have
\begin{align*}
 \Vert \alpha\star \beta \Vert_\zeta \le \frac{4\pi}{\zeta}\Vert \alpha\Vert_\zeta \Vert \beta\Vert_\zeta,
\end{align*}
 see \cite[Proposition 4]{bonckaert2008a}.
The Laplace transform (along the positive real axis)
\begin{align*}
 \mathcal L(\alpha)(x) :=\int_0^\infty \alpha(u) e^{-u/x} du,
\end{align*}
is then well-defined for any $\alpha\in \mathcal G$. In fact, we have the following result. 
\begin{lemma}\cite[Proposition 3]{bonckaert2008a}\lemmalab{A1}
 The Laplace transform defines a linear continuous mapping, with operator norm $\Vert \mathcal L\Vert\le 1$, from $\mathcal G$ to the set of analytic functions on a local sector $S_{\pi/2+\theta}\cap B(R_0)$ for $R_0>0$ sufficiently small. Moreover, 
 \begin{align}
  \mathcal L(\alpha \star \beta)(x) &= \mathcal L(\alpha)(x)\mathcal L(\beta)(x),\nonumber
  \end{align}
  and
  \begin{align}
  x^2\frac{d}{dx}\mathcal L(\alpha)(x)&=\mathcal L(u\alpha)(x),\eqlab{x2Lx}
 \end{align}
 with $u\alpha$ being the function $u\mapsto u\alpha(u)$, 
for every $\alpha,\beta\in \mathcal G$.
\end{lemma}
\begin{proof}
This is straightforward, see also further details in \cite{bonckaert2008a}.
\end{proof}

Following \eqref{x2Lx}, we are now led to write the left hand side of \eqref{conj} with $y=Y(x)$ as $[u-i\omega]\Phi(u)$ with $\Phi = \mathcal B(Y)$. 
To set up the associated right hand side, we need to deal with the nonlinearity $f(x,Y(x))$. This is described in \cite[Proposition 5]{bonckaert2008a}:
\begin{lemma}
Write $f$ as the convergent series $f(x,y)=\sum_{n=1}^\infty f_{n}(x) y^n$, $f_n(x):=\sum_{m=1}^\infty f_{mn}x^m$ and let $F_n(w)$ be the Borel transform of $f_n(x)$. Then $F_n\in \mathcal G$ for each $n$. Fix $C_0>0$ and for large values of $\zeta$, recall \eqref{normG}, consider $\alpha\in \mathcal G$ with $\Vert \alpha\Vert_\zeta\le C_0$. Then $\alpha\mapsto f^*(\alpha)$ defined by
\begin{align*}
f^*(\alpha)(u):=\sum_{n=1}^\infty F_n(u) \star \alpha(u)^{\star n},
\end{align*}
which converges in $\mathcal G$, is differentiable and satisfies the following estimates
\begin{align}
 \Vert f^*(\alpha)\Vert_{\zeta}\le C_1,\quad 
 \Vert D(f^*)(\alpha)\Vert_{\zeta}\le \zeta^{-1} C_1,\eqlab{fstarest}
\end{align}
for some constant $C_1>0$ depending only on $f$ and $C_0$. 
Moreover,
\begin{align}
 \mathcal L(f^*(\alpha))(x) = f(x,\mathcal L(\alpha)(x)).\eqlab{Lf}
\end{align}

\end{lemma}
\begin{proof}
The estimates in \eqref{fstarest} follow from simple estimates, see \cite{bonckaert2008a}. The last equality \eqref{Lf} is obtained by first working on the monomials $f_{mn} y^m x^n$, where the result is easy. Using the uniform continuity of $\mathcal L$, we obtain the result. 
\end{proof}
Following \eqref{Lf}, we are therefore finally led to consider 
\begin{align}
 \Phi(u) = [u-i\omega]^{-1} f^*(\Phi)(u).\eqlab{borelEq}
\end{align}
where $\Phi$ is the Borel transform of $Y$. The equation \eqref{borelEq} has the form of a fixed point equation. Notice specifically:
\begin{lemma}
The following holds \begin{align*}
 \vert u-i\omega\vert^{-1} \le \frac{1}{\textnormal{min}(\vert \omega \vert \cos \theta/2,\vert \omega\vert-R)},
\end{align*}
for $0<R<\vert \omega\vert$, $\theta\in(0,\pi)$ and
all $u\in \Omega$.
\end{lemma}
\begin{proof}
 Recall $\Omega=S_\theta \cup B(R)$, with $\theta\in (0,\pi)$. Suppose first that $u\in S_\theta$. Then $\vert u-i\omega\vert \ge \vert \omega \vert \cos \theta/2$ by simple trigonometry. For $u\in B(R)$ we have $\vert u-i\omega \vert \ge (\vert \omega\vert- R)>0$ for $R$ small enough.
\end{proof}
Using \eqref{fstarest}, it follows that there is some $M>0$, depending on $f$, $\theta$ and $R$ such that the right hand side of \eqref{borelEq} defines a contraction on the subset of $\mathcal G$ with $\Vert \cdot \Vert_\zeta\le M$ for $$\zeta>\frac{C_1}{\textnormal{min}(\vert \omega \vert \cos \theta/2,\vert \omega\vert-R)},$$ large enough. Consequently, by Banach's fixed point theorem there is a unique solution $\Phi\in \mathcal G, \Vert \Phi \Vert_\zeta\le M$, solving \eqref{borelEq}. By applying the Laplace transform we obtain the desired solution
\begin{align*}
 Y(x) := \mathcal L(\Phi)(x).
\end{align*}
 of \eqref{conj},
using \eqref{x2Lx} and \eqref{Lf}. The function $Y$ is defined on the domain $S_{\pi+\theta/2}\cap B(R_0)$ and has the properties specified by \lemmaref{A1}. This completes the proof of \lemmaref{center}.

\response{
\section{On the smoothness of the functions \eqref{T4int}}\applab{T4}
In this section, we consider
\begin{align}
Q(r):=\int_{r^2}^{1} s^{-3} r^3 T_4(s^2,s^{-3}r^3) ds^2.\eqlab{T4int2}
\end{align}
for every $r>0$, with $T_4$ being $C^\infty$ on the domain $[0,1]^2$ and show the following:
\begin{lemma}\lemmalab{QProp}
 Let $N\ge 3$ be as in \lemmaref{k3transformation} and set $L=1+2\lfloor \frac{N}{3}\rfloor$. Then $Q$ in \eqref{T4int2} has a $C^{L}$-smooth extension to $r=0$. 
\end{lemma}
This will in turn prove \lemmaref{P3Mapping} item (\ref{Fremaindersmoothness}). 
 To prove \lemmaref{QProp}, we use \lemmaref{T4prop}, which we restate here for convinience:
\begin{lemma}\lemmalab{T4prop2}
The following holds 
 \begin{align}
 \frac{\partial^{5l-2} }{\partial r_3^{2(3l-1)} \partial \epsilon_3^{2l-1}}T_4(0,0)=0,\eqlab{condres3}
\end{align}
for all $1\le l\le \lfloor \frac{N}{3}\rfloor$.
\end{lemma}

We write $T_4$ in an expansion similar to $f$ in the proof of \lemmaref{k3transformation}:
\begin{align*}
 T_4(r_3^2,\epsilon_3) = T_{41}(r_3^2,\epsilon_3)+T_{42}(r_3^2,\epsilon_3) +  r_3^{2N}\epsilon_3^{N} T_{43}(r_3^2,\epsilon_3)),
\end{align*}
where 
\begin{align*}
 T_{41}(r_3^2,\epsilon_3) = \sum_{n=0}^{N-1} T_{41n}(r_3^2) \epsilon_3^n,\quad T_{42}(r_3^2,\epsilon_3) = \sum_{n=0}^{N-1} T_{42n}(\epsilon_3) r_3^{2n}.
\end{align*}
For the expansion to be unique, we let 
\begin{align}\eqlab{condT41}
T_{42n}(\epsilon_3)= \mathcal O(\epsilon_3^{2N}),
\end{align}
for each $0\le n\le N-1$. Then
\begin{align}
 T_{41n}(r_3^2) = \frac{1}{n!}\frac{\partial^n}{\partial \epsilon_3^n}T_{4}(r_3^2,0),\eqlab{T41n}
\end{align}
for $0\le n\le N-1$.
We split $Q=\sum_{n=0}^{N-1} Q_{1n}+\sum_{n=0}^{N-1} Q_{2n}+Q_3$ in a similar way:
\begin{align}
 Q_{1n}(r)&:= \int_{r^2}^1 s^{-3(1+n)} r^{3(1+n)} T_{41n}(s^2) ds^2,\eqlab{Q1n}\\
 Q_{2n}(r)&:= \int_{r^2}^1 s^{-3+2n} r^3 T_{42n}(s^{-3}r^3)ds^2,\nonumber\\
 Q_3(r)&=\int_{r^2}^1 s^{-3-N} r^{3(1+N)} T_{43}(s^2,s^{-3}r^3)ds^2.\eqlab{Q3r}
\end{align}
For $Q_{2n}$ we will use a variable substitution:
\begin{align*}
 u^3 = s^{-3} r^3.
\end{align*}
This gives
\begin{align}
 Q_{2n}(r) = \frac{2}{3}\int_{r^3}^1 u^{-2(1+n)} r^{2(1+n)} T_{42n}(u^3) du^3.\eqlab{Q2n2}
\end{align}
This substitution corresponds to integration with respect to $\epsilon_3$ rather than $r_3^2$. We used a similar substitution for $A_1$, recall \eqref{Qexpr0}.

For the purpose of the analysis of $Q_{1n}$ and $Q_3$,
we first consider general integrals of similar form
\begin{align}
 \widetilde Q(r):=\int_{r^2}^1 s^{-q} r^p \widetilde T(s^2,s^{-3} r^3) ds^2. \eqlab{Qr}
\end{align}
for all $r>0$, $q,p\in \mathbb N$,
with $\widetilde T\in C^\infty$. 
\begin{lemma}\lemmalab{C0est}
Define 
\begin{align*}
 I:=p-q+1,
\end{align*}
and suppose that $I\ge 0$ and $q\ne 2$. Then $\widetilde Q(r)= \mathcal O(r^{\min(I+1,p)})$.
\end{lemma}
\begin{proof}
 Follows from a direct calculation:
 \begin{align*}
\widetilde Q(r)=\int_{r^2}^1 s^{-q} r^p \mathcal O(1) ds^2 =  \frac{2}{2-q} [s^{-q+2}r^p]_{s^2=r^2}^1\mathcal O(1)=\mathcal O(r^{p-q+2},r^p),
 \end{align*}
using a uniform bound on $\widetilde T$. 
\end{proof}

We say that \eqref{Qr} has \textit{index} $I$.

\begin{lemma}\lemmalab{CIfunction}
Consider \eqref{Qr} and suppose that $\widetilde Q$ has index $I\ge 1$, $q>0$, $p\in \mathbb N$, but $q\ne 2$ and that $\widetilde T\in C^\infty$. Then $Q(r)=\mathcal O(r^{\min(I+1,p)})$ extends to $r=0$ as a $C^{I}$-function.
\end{lemma}
\begin{proof}
By the Leibniz rule of differentiation we have
 \begin{align}
  \widetilde Q'(r) &=-2r^{I} \widetilde T(r^2,1)  +p\int_{r^2}^1 s^{-q} r^{p-1} \widetilde T(s^2,s^{-3} r^3)ds^2\nonumber \\
  &+3\int_{r^2}^1 s^{-3-q} r^{p+2} \frac{\partial \widetilde T}{\partial \epsilon_3} (s^2,s^{-3}r^3)) ds. \eqlab{Q1}
 \end{align}
The second and third terms are integrals of the form \eqref{Qr} of index $I-1\ge 0$ and $\widetilde Q'(r)=\mathcal O(r^{\min(I-1,p-1)})$. $\widetilde Q$ then has a $C^1$ extension to $r=0$ if $p\ge 2$. If $p=1$ then by assumption $I\ge 1$, we have $q= 1$ and $I=1$. Consequently, the second the term is continuous at $r=0$ with value 
\begin{align*}
\int_{0}^1 s^{-q} \widetilde T(s^2,0)ds^2.
\end{align*}
It is easy to complete the result by induction.

\end{proof}
We now turn to the smoothness of $Q_{1n}$ and $Q_3$. For $Q_3$, we notice that it is of the form \eqref{Qr} with index
\begin{align*}
 I=  3(1+N)-3-N+1=2N+1.
\end{align*}
By \lemmaref{CIfunction} $Q_3$ is $C^{2N+1}$. 

Next, for $Q_{1n}$. It has index $I_0=1$ for each $n$. However, we now show that we can increase the index by applying integration by parts:
\begin{equation}
 \eqlab{intparts}
%
\begin{aligned}
 \frac{2-3(1+n)}{2}Q_{1n}(r) &= r^{3(1+n)}T_{41n}(1)-r^2 T_{41n}(r^2)\\
 &-\int_{r^2}^1 s^{-3(1+n)+2} r^{3(1+n)} T_{41n}'(s^2)ds^2.
\end{aligned}
\end{equation}
The integral on the right hand side now has index $I_1=I_0+2=3$. We can proceed in this way $ \lfloor \frac{N}{3}\rfloor$-number of times, since if $-3(1+n)+2m= -2$ (in which case the integration gives $\log s^2$) then we have
\begin{align*}
 T_{41n}^{(m)}(0) =0.
\end{align*}
This follows from \eqref{T41n}
and \eqref{condres3}. Hence, after $\lfloor \frac{N}{3}\rfloor$-many steps, we obtain a (finite) sum of smooth functions and an integral of the form \eqref{Qr} with index $I_{\lfloor \frac{N}{3}\rfloor} = 1+2\lfloor \frac{N}{3}\rfloor$. Consequently, by \lemmaref{CIfunction}, $Q_{1n}$ is $C^L$ with $L=1+2\lfloor \frac{N}{3}\rfloor$ for each $n$. 

\begin{remark}\remlab{bookkeeping}
 Obtaining a $C^k$-version of \thmref{mainTP2}, is complicated by the fact that upon applying integration by parts, we lose degrees of smoothness of the boundary terms  (see first two terms in \eqref{intparts} after the first step). In fact, the smoothness of these terms depends upon $n$ to complicate matters further. By assuming that $\mu\in C^\infty$, we avoid this cumbersome bookkeeping. In the $C^\infty$-case, the degree of smoothness of the integrals $Q$ are only determined by the ``resonances'' \eqref{condres3}. We leave further details on the $C^k$-case to the interested reader.
\end{remark}

Next, we turn to $Q_{2n}$ in \eqref{Q2n2}.
For this purpose we consider a general integral of a similar form
\begin{align}
 \widetilde Q(r) = \int_{r^3}^1 u^{-q} r^{p} \widetilde T(u^3) du^3,\eqlab{tildeQr}
\end{align}
with $\tilde T\in C^\infty$ and define an index in a similar same way $I=p-q+1$. Then by proceeding as in the proof of \lemmaref{C0est}, we have
\begin{align*}
 \widetilde Q(r) = \mathcal O(r^{\min (I+2,p)}),
\end{align*}
provided that $q>0$ but $q\ne 3$. We also have by the Leibniz rule of differentiation that
\begin{align*}
 \widetilde Q'(r) = -r^{I} \widetilde T(r^3)+p\int_{r^3}^1 u^{-q}r^{p-1} \widetilde T(u^3) du^3,
\end{align*}
with the second integral having index $I-1$. Therefore we have the following.
\begin{lemma}\lemmalab{CIfunction2}
Suppose that $\widetilde Q$ in \eqref{tildeQr} has index $I\ge 1$. Then it has a $C^I$-smooth extension to $r=0$.
\end{lemma}
The proof is similar to \lemmaref{CIfunction} and therefore left out.

Now, $Q_{2n}$ has index $I_0=1$ for all $n$, but we can, as for $Q_{1n}$, apply integration by parts on \eqref{Q2n2} to increase the index:
\begin{align*}
 \frac{3-2(1+n)}{3} Q_{2n}(r) &= r^{2(1+n)} T_{42n}(1)-r^3 T_{42n}(r^3)\\
 &-
 \int_{r^3}^1 u^{-2(1+n)+3} r^{2(1+n)} T_{42n}'(u^3) du^3.
\end{align*}
Here the integral on the right hand side now has index $I_1=1+3$. Continuing in this way $\lfloor \frac{N}{3}\rfloor$-number of times,, we obtain an integral with index $I_{\lfloor \frac{N}{3}\rfloor}=3\lfloor \frac{N}{3}\rfloor+1$. Then by \lemmaref{CIfunction2}, we finally conclude that
\begin{align*}
 \sum_{n=0}^{N-1} Q_{1n}\in C^L,\,\sum_{n=0}^{N-1} Q_{2n}\in C^{3\lfloor \frac{N}{3}\rfloor+1},\,Q_3\in C^{2N+1}.
\end{align*}
Seeing that $L=2\lfloor \frac{N}{2}\rfloor+1\le 3\lfloor \frac{N}{3}\rfloor+1\le 2N+1$, we have that $Q$ in \eqref{T4int2} is $ C^L$ for $r\ge 0$ as claimed. This finishes the proof of \lemmaref{QProp}.
}

\end{document}